\documentclass[]{article}
\usepackage{amsmath,amsthm,amssymb,amsfonts,amscd} 
\usepackage{color}
\usepackage{comment}
\usepackage{paralist}
\usepackage{mathrsfs}
\usepackage{hyperref} 
\usepackage{setspace}
\usepackage{xcolor}
\usepackage{titlesec}

\titleformat{\subsection}{\normalfont\scshape\filcenter}{\thesubsection}{1em}{}
\titleformat{\subsubsection}[runin]{\normalfont\bfseries}{\thesubsubsection. }{.2em}{}[.]

\newcommand\thefontsize[1]{{#1 The current font size is: \f@size pt\par}}

\newtheorem{lem}{Lemma}[section]

\newtheorem{remark}{Remark}[section]
\newtheorem{cor}[lem]{Corollary}

\newtheorem{THM}{Theorem}

\newtheorem{prop}[lem]{Proposition}

\newcommand{\R}{{\mathbb R}}
\newcommand{\N}{{\mathbb{N}}}

\newcommand{\mpp}{\mathcal{P}}

\newcommand{\p}{\partial}

\DeclareMathOperator{\supp}{supp}

\DeclareMathOperator{\Lip}{Lip}

\numberwithin{equation}{section}

\title{Finiteness Principles for Smooth Convex Functions}
\author{Marjorie K. Drake\thanks{This material is based upon work supported by the National Science Foundation under Award
No. 2103209.}}
\date{}

\begin{document}

\maketitle

\begin{abstract}
Let $E \subset \R^n$ be a compact set, and $f:E \to \R$. How can we tell if there exists a convex extension $F \in C^{1,1}(\R^n)$ of $f$, i.e. satisfying $F|_E = f|_E$? Assuming such an extension exists, how small can one take the Lipschitz constant $\Lip(\nabla F): = \sup_{x,y \in \R^n, x \neq y} \frac{|\nabla F(x) - \nabla F(y)|}{|x-y|}$? We provide an answer to these questions for the class of strongly convex functions by proving that there exist constants $k^\# \in \N$ and $C>0$ depending only on the dimension $n$, such that if for every subset $S \subset E$, $\#S \leq k^\#$, there exists an $\eta$-strongly convex function $F^S \in C^{1,1}(\R^n)$ satisfying $F^S|_S=f|_S$ and $\Lip(\nabla F^S) \leq M$, then there exists an ${\frac{\eta}{C}}$-strongly convex function $F \in C^{1,1}_c(\R^n)$ satisfying $F|_E = f|_E$, and $\Lip(\nabla F) \leq C M^2/\eta$. Further, we prove a Finiteness Principle for the space of convex functions in $C^{1,1}(\R)$ and that the sharp finiteness constant for this space is $k^\#=5$. 
\end{abstract}

\section{Introduction}

Let $C_c^{1,1}(\R^n)$ be the space of convex, differentiable functions with Lipschitz continuous gradient. We say that a function $F:\R^n\to\R$ is {\em $\eta$-strongly convex}, for $\eta \geq 0$, if the function $F(x)-\frac{\eta}{2}|x|^2$ is convex. 
Let $E \subset \R^n$ be compact and $f:E \to \R$. In this paper, we provide an answer to the following questions: Under what conditions on $\eta$ and $f$ does there exists an $\eta$-strongly convex function $F \in C_c^{1,1}(\R^n)$ that is an extension of $f$, i.e., satisfying $F|_E=f|_E$? Assuming such an extension exists, how small can one take the Lipschitz constant $\Lip(\nabla F)$ for an $\eta$-strongly convex extension $F$ of $f$? Recall the Lipschitz constant of a function $G:\R^n \to \R^d$ is defined as $\Lip(G): = \sup_{x,y \in \R^n, x \neq y} \frac{|G(x) - G(y)|}{|x-y|}$. We prove the following result:

\begin{THM} \label{scthm}
    Let $E \subset \R^n$ be compact, the constants $\eta, M$ satisfy $M>\eta>0$, and the function $f:E \to \R$. There exist ${k^\#} \in \N$ and $C>0$ depending only on the dimension $n$ such that the following holds: Suppose that for all $S \subset E$ satisfying $\#S \leq {k^\#}$, there exists an $\eta$-strongly convex function $F^S \in C^{1,1}_c(\R^n)$ satisfying $F^S|_S=f|_S$ and $\Lip(\nabla F^S) \leq M$. Then for any $p,q \in (1, \infty)$ satisfying $\frac{1}{p}+\frac{1}{q}=1$, there exists an ${\eta /p^2}$-strongly convex function $F \in C^{1,1}_c(\R^n)$ satisfying $F|_E = f|_E$, and $\Lip(\nabla F) \leq C_1 q^2M^2/\eta$.
\end{THM}

Fix a constant $C_0>C_1q^2$ and $p, q \in (1, \infty)$ satisfying $\frac{1}{p}+\frac{1}{q}=1$; suppose the hypotheses of Theorem \ref{scthm} are satisfied by $E,f, M$, and $\eta \in (\frac{C_1q^2}{C_0}M ,M)$. Then Theorem \ref{scthm} produces an $\eta/p^2$-strongly convex extension of $f$, $F \in C^{1,1}_c(\R^n)$ satisfying $\Lip(\nabla F) \leq C_0 M$. But if instead the hypotheses are satisfied by $E,f,M$, and $\eta$ much smaller than $M$ ($\eta \in [0, \frac{C_1q^2}{C_0}M)$), we expect this theorem is not optimal. We conjecture that satisfaction of the hypotheses of Theorem \ref{scthm} ensures the existence of a strongly convex extension of $f$, i.e. $F \in C^{1,1}_c(\R^n)$ satisfying $\Lip(\nabla F) \leq C M$, where $C$ depends on $n, p$, and $q$, but not on $\eta$ or $M$.
Indeed this is true in dimension $n=1$, which is our second result:


\begin{THM} \label{1dthmsc}
Let $E \subset \R$ be compact, the constants $\eta, M$ satisfy $M>\eta \geq 0$, and the function $f:E \to \R$. Suppose for every $S \subset E$ satisfying $\#S \leq {k_1^\#}=5$, there exists an $\eta$-strongly convex function $F^S \in C^{1,1}_c(\R)$ satisfying $F^S|_S=f|_S$ and $\Lip(\nabla F^S) \leq M$. Then there exists an $\eta$-strongly convex function $F \in C^{1,1}_c(\R)$ satisfying $F|_E = f|_E$ and $\Lip (\nabla F) \leq 5M$.
\end{THM}

\begin{remark}
    In Theorem \ref{1dthmsc}, no constant smaller than ${k_1^\#}=5$ will suffice (the \emph{sharp} finiteness constant for $C^{1,1}_c(\R)$ is ${k_1^\#}=5$). To see this, consider the following example: Let $E \subset \R$ be $E:= \{ -2, -1, 0, 1, 2 \}$; for $x \in E$, let $f(x) := |x|$. For every set $S \subset E$ satisfying $\#S \leq 4$, one can construct a convex function $F^S \in C^{1,1}_c(\R)$ satisfying $F|_S = f|_S$, but any convex extension of $f$, $F: \R \to \R$ must satisfy $F(x) = |x|$ for $x \in [-2,2]$, which is not differentiable at $x=0$. Thus, we must have ${k_1^\#}>4$.
\end{remark}

Our results are the first attempt to understand the constrained interpolation problem for \emph{convex} functions in $C^{1,1}_c(\R^n)$.\footnote{The constrained interpolation problem where the interpolating function is required to be non-negative has been studied by C. Fefferman, A. Israel, and K. Luli in \cite{arie10}, and K. Luli, and F. Jiang in \cite{jilu1} and \cite{jilu2}. } We build on techniques used to understand whether a function has a smooth extension despite obstacles to their direct application.

Let $\mathbb{X}(\R^n) \subset C(\R^n)$ be a complete semi-normed space of continuous functions. Given a compact set $E \subset \R^n$ and a function $f:E \to \R$, how can we tell if there exists $F \in \mathbb{X}(\R^n)$ extending $f$, that is satisfying $F|_E=f|_E$? In \cite{shv7,shv6}, Pavel Shvartsman answered this question for the linear space $\mathbb{X}(\R^n) = C^{1.1}(\R^n)$ through a \emph{Finiteness Principle}.

We say that there is a \emph{Finiteness Principle for $\mathbb{X}(\R^n)$} if there exist $k \in \N$, $C>0$ depending on $\mathbb{X}(\R^n)$ such that given $E \subset \R^n$ finite and $f:E \to \R$, if we assume for every $S \subset E$ satisfying $\#S \leq k$, there exists $F^S \in \mathbb{X}(\R^n)$ satisfying $F^S|_S = f|_S$ and $\|F^S\|_{X(\R^n)} \leq 1$, then there exists $F \in \mathbb{X}(\R^n)$ such that $F|_E = f|_E$ and $\|F\|_{\mathbb{X}(\R^n)} \leq C$.

Further, P. Shvartsman proved that the {sharp} finiteness constant for $C^{1,1}(\R^n)$ is $k = 3\cdot 2^{n-1}$, and conjectured with Yuri Brudnyi that Finiteness Principles for the linear spaces $C^{m}(\R^n)$ and $C^{m-1,1}(\R^n)$ would hold in \cite{brshv1,brshv2}. In \cite{f1,f2} Charles Fefferman proved Finiteness Principles for these spaces ($C^{m-1,1}(\R^n)$ and $C^{m}(\R^n)$). 

Theorems \ref{scthm} and \ref{1dthmsc} are progress toward the proof of a Finiteness Principle for the non-linear space of smooth convex function $C^{1,1}_c(\R^n)$. Our hope is this work and the continued study of finiteness principles for smooth convex functions allow the development of algorithms for constructing smooth, convex extensions of a function (or its approximation) analogous to the work by C. Fefferman and Boaz Klartag in \cite{f7,f8} for $C^m(\R^n)$. 

Our proofs of Theorems \ref{scthm} and \ref{1dthmsc} rely on an inequality relating the \emph{jets} of a convex function in $C^{1,1}_c(\R^n)$. Let $\mpp$ be the space of real-valued affine (degree one) polynomials. For $F \in C^1(\R^n)$, we define the jet of $F$ at $x$, $J_xF \in \mpp$ as $J_xF(y):= F(x) + \langle \nabla F(x), y-x \rangle$.
Let the function $F \in C_c^{1,1}(\R^n)$ be convex; as a consequence of Taylor's inequality, 
\begin{align*}
        F(x) - J_yF(x) \geq \frac{1}{2\Lip(\nabla F)}|\nabla F(x) - \nabla F(y)|^2 \quad \quad (x,y \in \R^n). 
\end{align*}
In Section \ref{sec:cw11}, we prove this inequality. In \cite{az1,az2}, Daniel Azagra, Erwan Le Gruyer, and Carlos Mudarra  proved a partial converse to this inequality, criteria for convex $C^{1,1}$-extension of degree one polynomials defined on a closed set $E \subset \R^n$, which is a key component of our proofs:

\begin{THM}[D. Azagra, E. Le Gruyer, and C. Mudarra \cite{az2}, Theorem 2.4] \label{azagra0}
    Let $E \subset \R^n$ be closed and the polynomials $(P_x)_{x \in E} \subset \mpp$ satisfy for all $x, y \in E$, 
    \begin{align}
    P_x(x) - P_y(x) \geq \frac{1}{2M}|\nabla P_x - \nabla P_y|^2. \label{cw11}    
    \end{align}
    Then there exists a convex function $F \in C^{1,1}_c(\R^n)$ satisfying $J_xF = P_x$ for all $x \in E$ and $\Lip(\nabla F) \leq M$. 
\end{THM}

We now give a sketch of the proof of Theorem \ref{scthm}. We write $c, C, C',$ etc. to denote constants dependent only on the dimension $n$. By appealing to the Arzel\`a-Ascoli Theorem, we reduce to the case $E \subset \R^n$ finite. In Proposition \ref{scprop}, we prove that if $(P_x)_{x \in E}$ satisfy \eqref{cw11} and 
\begin{align}
    P_x(y) + \frac{\eta}{2}|y-x|^2 \leq P_y(y) \quad \text{ for all }x,y \in E, \label{scintro}
\end{align}
then the conclusions of Theorem \ref{azagra0} hold for a \emph{strongly} convex function $F \in C^{1,1}_c(\R^n)$ satisfying $\Lip(\nabla F) \leq CM$. Thus, given $f:E \to \R$, we aim to find $(P_x)_{x \in E} \subset \mpp$ satisfying $P_x(x) =f(x)$ and the inequalities (\ref{cw11}) and \eqref{scintro} with uniform constants $M$ and $\eta$ for all $x,y \in E$. To do this, we introduce an approximation of the set of prospective jets of strongly-convex $C^{1,1}_c(\R^n)$ extensions of $f$. 
For $x \in E$, $\eta>0$, let $\Gamma^E_\eta(x) \subset \mpp$ be
\begin{align*}
    \Gamma^E_\eta(x):=\{ P \in \mpp: P(x) = f(x) \text{ and } P(y) + \frac{\eta}{2}|y-x|^2\leq f(y) \text{ for all } y \in E \setminus \{x\}\}.
\end{align*}
Immediately, we see for an $\eta$-strongly convex extension $F \in C^{1,1}_c(\R^n)$ satisfying $F|_E = f|_E$ and $\Lip (\nabla F) \leq M$, we have $J_xF \in \Gamma^E_\eta(x)$ for all $x \in E$. We prove we can choose $(P_x)_{x \in E} \subset \mpp$ so that 
\begin{align}
    &P_x \in \Gamma^E_\eta(x) &&(x \in E), \text{ and} \label{j1} \\
    &\sup_{x,y \in E, x \neq y} \left\{ \frac{|\nabla P_x - \nabla P_y|}{|x-y|} \right\} \leq C'M.    \label{j2}
\end{align}
Together \eqref{j1} and \eqref{j2} imply
\begin{align*}
    f(y) - P_x(y) \geq \frac{\eta}{2}|y-x|^2 \geq \frac{\eta}{2(C'M)^2}|\nabla P_x - \nabla P_y|^2 \quad \quad (x, y \in E).
\end{align*}
Hence, this choice of $(P_x)_{x \in E}$ satisfies (\ref{cw11}) with a constant $(C'M)^2/\eta$, and we can apply Proposition \ref{scprop}.

To prove we can choose $(P_x)_{x \in E}$ satisfying \eqref{j1} and \eqref{j2}, we use Helly's Theorem, (following P. Shvartsman in \cite{shv7}, \cite{shv6}, and C. Fefferman in \cite{f1}, \cite{f2}) and a Finiteness Principle for Smooth Selection proved by C. Fefferman, Arie Israel, and Kevin Luli in \cite{arie9}; see also C. Fefferman and P. Shvartsman's results in \cite{fshv}.

\begin{THM}[Helly, see e.g. \cite{helly}] \label{helly}
    Let $\mathcal{J}$ be a finite family of convex subsets of $\R^d$, and suppose every $(d+1)$ elements of the family has non-empty intersection. Then the entire family has non-empty intersection. If $\mathcal{J}$ is infinite, the sets must also be compact for the result to follow.
\end{THM}

For $D \geq 1$, let $C^{0,1}(\R^n, \R^D)$ denote the Banach space of all $\R^D$-valued Lipschitz functions $F$ on $\R^n$, for which the norm $\|F\|_{C^{0,1}(\R^n, \R^D)} = \sup_{x \in \R^n} \{ | F(x)| \}+ \Lip( F)$, is finite.

\begin{THM}[C. Fefferman, A. Israel, and K. Luli (Theorem 3(B) of \cite{arie9})] \label{selthm0}
    There exist $k^\#_s=k^\#_s(n,D) \in \N$ and $C^\#=C^\#(n,D)>0$ such that the following holds: Let $E \subset \R^n$ be arbitrary. For each $x \in E$, let $K(x) \subset \R^D$ be a closed convex set. Suppose that for each $S \subset E$ with $\#S \leq k^\#_s$, there exists $F^S \in C^{0,1}(\R^n,\R^D)$ with norm $\|F^S\|_{C^{0,1}(\R^n,\R^D)} \leq 1$, such that $F^S(x) \in K(x)$ for all $x \in S$. Then there exists $F \in C^{0,1}(\R^n,\R^D)$ with norm $\|F^S\|_{C^{0,1}(\R^n,\R^D)} \leq C^\#$, such that $F(x) \in K(x)$ for all $x \in E$.
\end{THM}
 
We assume the following Finiteness Hypothesis: for every $S \subset E$, $\#S \leq {k^\#}$, there exists an $\eta$-strongly convex function $F^S \in C^{1,1}_c(\R^n)$ satisfying $F^S|_S=f|_S$ and $\Lip(\nabla F^S) \leq M$, with ${k^\#} = k^\#_s(n+2)+1$ and $k^\#_s= k^\#_s(n,n)$ from Theorem \ref{selthm0}. Using this hypothesis, we can apply Helly's Theorem to show the hypotheses of Theorem \ref{selthm0} are satisfied for the family of convex sets $(K(x) = \{ \nabla P: P \in \Gamma^E_\eta(x)\})_{x \in E}$ in $\R^n$. Thus, we can apply Theorem \ref{selthm0} to produce a Lipschitz selection $G \in C^{0,1}(\R^n,\R^n)$ from $(K(x))_{x \in E}$ satisfying $\Lip(G) \leq C'M$. For $x \in E$, we let $P_x(x) :=f(x)$ and $\nabla P_x := G(x)$, and as promised, $(P_x)_{x \in E}$ satisfies \eqref{j1} and \eqref{j2}.

This concludes our sketch of the proof of Theorem \ref{scthm}. The rest of the paper is organized as follows: In Section \ref{sec:tech}, we adapt Theorems \ref{azagra0} and \ref{selthm0} to our setting and analyze sets approximating the set of jets of smooth convex extensions of the function $f$, including ${\Gamma}^E_\eta(x)$. In Section \ref{sec:scthm}, we prove Theorem \ref{scthm} for dimension $n\geq1$. In Section \ref{sec:1dthm}, we detail technical estimates that hold only in dimension $n=1$, and prove Theorem \ref{1dthmsc}.

\subsection{Acknowledgments}
The author is grateful to Arie Israel for providing valuable comments on an early draft of this paper and the National Science Foundation for its generous support.

\subsection{Notation}
Let $E \subset \R^n$, $f:E \to \R$. We use the following notation:
\begin{align*}
&|x| := |x|_2= (|x_1|^2 + \dots + |x_n|^2)^{1/2}. &&(x = (x_1, \dots, x_n) \in \R^n ); \\
& B(y, R):= \{ x \in \R^n: |x-y| < R \} &&(y \in \R^n, \; R \geq 0 ); \\
&D^f_{xy}:=\frac{f(y) - f(x)}{y-x} 
&&(x,y \in E, x \neq y, \; E \subset \R).
\end{align*}

Let $\Omega \subset \R^n$ be a \emph{domain} (i.e., a non-empty, connected open set), and let the vector-valued function $F:\Omega \to \R^D$. The Lipschitz constant of the function $F$ is
\begin{align*}
&\Lip(F; \Omega): = \sup_{x,y \in \Omega, x \neq y} \frac{|F(x) - F(y)|}{|x-y|}.
\end{align*}
Where the domain $\Omega$ is evident, we write $\Lip(F)$ in place of $\Lip(F;\Omega)$.

For $m=0$ or $m=1$, let $C^m(\Omega)$ denote the Banach space of real-valued $C^m$ functions $F$ on $\Omega$ for which the norm 
$$
\|F\|_{C^{m}(\Omega)} = \sup_{x \in \Omega} \max_{|\alpha| \leq m} |\p^\alpha F(x)|
$$
is finite, and $C^{m,1}(\Omega)$ denote the Banach space of real-valued $C^m$ functions $F$ on $\Omega$ with Lipshitz continuous gradient for which the norm 
$$
\|F\|_{C^{m,1}(\Omega)} = \|F\|_{C^{m}(\Omega)} + \Lip(\nabla^m F;\Omega) 
$$
is finite. 

For $D \geq 1$, let $C^{m,1}(\R^n, \R^D)$ denote the Banach space of vector-valued $C^m$ functions $F$ on $\R^n$, for which the norm
$$
\|F\|_{C^{m,1}(\R^n, \R^D)} = \sup_{x \in \R^n} \max_{|\alpha| \leq m} |\p^\alpha F(x)| + \Lip(\nabla^m F;\R^n)
$$
is finite.

Let $C^{m,1}_{loc}(\Omega)$ denote the space of functions $F$ on $\R^n$ satisfying $\|F\|_{C^{m,1}(\Omega')} < \infty$ for all bounded open sets $\Omega' \subset \subset \Omega$. 

Let $\Omega \subset \R^n$ be a convex domain. Let $C_c^{1,1}(\Omega) \subset C^{1,1}_{loc}(\Omega)$ denote the space of convex, differentiable functions with Lipschitz continuous gradient. 

Let $\mpp$ be the space of real-valued affine (degree one) polynomials. For $F \in C^1(\R^n)$, we define the jet of $F$ at $x$, $J_xF \in \mpp$ as 
$$J_xF(y):= F(x) + \langle \nabla F(x),y-x\rangle.$$ 
For each $x \in \R^n$, the jet product $\odot_x$ on $\mpp$ is
defined by
\[
P \odot_{x} Q:=J_{x}(P \cdot Q) \quad(P, Q \in \mpp).
\]
Let $\mathcal{R}_x = (\mpp, \odot_x)$ be the ring of $1$-jets of functions at $x\in \R^n$.

Let $E \subset \R^n$ and $P_x \in \mathcal{R}_x$ for all $x \in E$; then we say $(P_x)_{x \in E} \subset \mpp$ is a Whitney field on $E$. Let $Wh(E)$ be the set of all Whitney fields on $E$.

For $\gamma_x \in \mathcal{R}_x$, $\gamma_y \in \mathcal{R}_y$, we will say $\gamma_x  \sim_M \gamma_y$ if the following inequalities are satisfied:
\begin{align}
\gamma_x(x) -\gamma_y(x) &\geq \frac{1}{2M} |\nabla \gamma_x - \nabla \gamma_y|^2 \label{a} \\ 
\gamma_y(y)-\gamma_x (y) &\geq \frac{1}{2M} |\nabla \gamma_x - \nabla \gamma_y|^2.  \label{b} 
\end{align}

We write $c, C, C',$ etc. to denote constants dependent only on the dimension $n$. 

\section{Technical Tools} \label{sec:tech}

Let $E \subset \R^n$ be compact, and let $f: E \to \R$. We now introduce certain convex subsets of $\mathcal{R}_x$ that reflect constraints on the jet of a convex extension of $f$. 
For $S \subset E$ and $x \in S$, let $\Gamma^0(x;f), \; \Gamma^S(x;f) $, and $\Gamma^S_\eta(x;f) \subset \mathcal{R}_x$ be 
\begin{equation}\label{Gamma_summary}
\begin{aligned}
&\Gamma^0(x;f) := \{ P \in \mathcal{R}_x: P(x) = f(x) \}, \text{ and}\\
&\Gamma^{S}(x;f) :=\{ P \in \Gamma^0(x;f): P(y) \leq f(y) \text{ for all } y \in S\}\\
&\Gamma_\eta^{S}(x;f)=\{ P \in \Gamma^0(x;f): P(y) + \frac{\eta}{2}|y-x|^2\leq f(y) \text{ for all } y \in S \setminus \{x\}\}\\
& \quad \quad \quad = \bigcap_{y \in S \setminus \{x\}} \{ P \in \Gamma^0(x;f): P(y) + \frac{\eta}{2}|y-x|^2 \leq f(y) \}.
\end{aligned}
\end{equation}
Where the function $f$ is evident, we will not write $f$; i.e., we write ${\Gamma}^0(x)$ in place of ${\Gamma}^0(x;f)$.

Any convex extension of the function $f$, $F: \R^n \to \R$, satisfies $\partial F(x) \subset \Gamma^E(x;f)$ for all $x \in E$, where $\partial F(x):= \{ \xi \in\R^n: F(y)\geq F(x)+\langle \xi ,y-x\rangle \text{ for all } y\in \Omega \}$ is the \emph{subdifferential of $F$ at $x$}.

The sets $\Gamma^0(x), \; \Gamma^S(x) $, and $\Gamma^S_\eta(x)$ are convex subsets of $\mathcal{R}_x$; this property is immediate for $\Gamma^0(x)$. To see the set ${\Gamma}^S(x)$ is convex, notice if $P \in \{ P \in \Gamma^0(x): P(y) \leq f(y) \}$ for $y \in S \setminus \{x\}$, then $P(x) = f(x)$ and $\nabla P$ satisfies the linear inequality $f(x) + \langle \nabla P, y-x \rangle \leq f(y)$. Hence, $\{ P \in \Gamma^0(x): P(y) \leq f(y) \}$ is convex, implying $\Gamma^S(x) = \bigcap_{y \in S \setminus \{x\}} \{ P \in \Gamma^0(x): P(y)  \leq f(y) \}$ is convex. Similarly, if $P \in \{ P \in \Gamma^0(x): P(y) + \frac{\eta}{2}|y-x|^2 \leq f(y) \}$ for $y \in S \setminus \{x\}$, then $P(x) = f(x)$ and $\nabla P$ satisfies the linear inequality $f(x) + \langle \nabla P, x-y \rangle + \frac{\eta}{2}|y-x|^2 \leq f(y)$, implying ${\Gamma}_\eta^S(x)$ is convex as an intersection of convex sets.

\begin{lem} \label{lemlip}
    Let $E \subset \R^n$ be compact and $f:E \to \R$. Suppose $\Gamma^E(x;f) \neq \emptyset$ for all $x \in E$. Then there exists a convex (and thus, locally Lipschitz) function $F: \R^n \to \R$ extending $f$.
\end{lem} 
\begin{proof}
    Let $F: \R^n \to \R$ be $F(x):= \sup_{y \in E} \{P_y(x): P_y \in \Gamma^E(y) \}$; then $F|_E=f|_E$ and as the supremum of convex functions, $F$ is convex. 
\end{proof}

\begin{lem} 
Let $E \subset \R^n$ be compact and $f: E \to \R$. Suppose $\Gamma^E_\eta(x;f) \neq \emptyset$ for all $x \in E$; then there exists an $\eta$-strongly convex function $F:\R^n \to \R$ extending $f$.
\end{lem}

\begin{proof}
   Suppose $\Gamma^E_\eta(x;f)  \neq \emptyset$ for all $x \in E$, then $\Gamma^E(x;g) \neq \emptyset$ for all $x \in E$, where $g(x): = f(x) - \frac{\eta}{2}|x|^2$. By Lemma \ref{lemlip} there exists a convex function $G: \R^n \to \R$ satisfying $G|_E=g|_E$. Thus, $F(x):=G(x) + \frac{\eta}{2} |x|^2$ is $\eta$-strongly convex and satisfies $F|_E=f|_E$. 
\end{proof}

\begin{lem} \label{necsc}
    Let $E \subset \R^n$ be compact and $f: E \to \R$. Let $S \subset E$ and $F \in C^{1,1}_c(\R^n)$ be an $\eta$-strongly convex function satisfying $F|_S =f|_S$. Then $J_xF \in \Gamma^S_\eta(x;f)$ for all $x \in S$.
\end{lem}
\begin{proof}
    Let $G: \R^n \to \R$ be $G(x):= F(x) - \frac{\eta}{2}|x|^2$. Because $F$ is $\eta$-strongly convex, $G$ is convex, implying for all $x,y \in S$, $J_xG(y) \leq G(y)$; equivalently,
    \begin{align*}
        F(x) - \frac{\eta}{2}|x|^2 + \langle \nabla F(x) - \eta x, y-x \rangle \leq F(y) - \frac{\eta}{2}|y|^2 \quad (x,y \in S).
    \end{align*}
    This reduces to $J_xF(y) +\frac{\eta}{2}|y-x|^2 \leq F(y)=f(y)$, implying $J_xF \in \Gamma^S_\eta(x;f)$ for all $x \in S$.
\end{proof}

\subsection{An Estimate on $1$-Jets of Convex Functions in $C_c^{1,1}$} \label{sec:cw11}

Recall that $B(a,r)$ is the Euclidean ball of radius $r$ centered at $a$: $B(a,r): = \{ x \in \R^n: |x-a| < r \}$. 
For a function $F \in C^{1,1}(B(a,2R))$, we use \underline{Taylor's inequality} to bound the difference between the function value and its jet evaluated at a point:
\begin{align}
|F(z) - J_xF(z)| \leq \frac{1}{2}{\Lip(\nabla F; B(a,2R))}|x-z|^2 \quad \quad (x,z \in B(a,2R)), \label{TI}
\end{align}
where $\Lip (\nabla F; B(a,2R)) := \sup_{x \neq y, x,y \in B(a,2R)} \left\{ \frac{|\nabla F(x) - \nabla F(y)|}{|x-y|}\right\}$.  We use \eqref{TI} in the following estimate on the jets of a convex function $F \in C^{1,1}_c(B(a,2R))$.

\begin{lem} \label{neclem}
    Let $F \in C^{1,1}_c(B(a,2R))$ be convex; then $J_xF \sim_{M} J_yF$ for all $x,y \in B(a,R)$, where $M=\Lip(\nabla F; B(a,2R))$.
\end{lem}

\begin{proof}
    We adapt the proof of Proposition 3.2 in \cite{az1}. For $F$ affine, the result is immediate. Suppose $F$ is not affine. Let $M: = \Lip(\nabla F; B(a,2R))$, and suppose there exist $x, y \in B(a,R)$ such that $J_xF \not\sim_M J_yF$. Then without loss of generality, we can assume
    \begin{align}
        F(x) - F(y) - \langle\nabla F(y), x-y \rangle <  \frac{1}{2M}|\nabla F(x) - \nabla F(y)|^2. \label{nothold}
    \end{align}
    By translation (by $y$) and subtraction of an affine function ($z \mapsto F(y)+ \nabla F(y) (z-y)$), we can assume $y = 0 \in B(a,R)$, $F(y) = 0$, and $\nabla F(y) = 0$. Because $F$ is convex, this implies $F(z) \geq 0$ for $z \in B(a,2R)$, and (\ref{nothold}) becomes
    \begin{align*}
        0 \leq F(x) < \frac{1}{2M}|\nabla F(x)|^2.
    \end{align*}
    In particular, $\nabla F(x)  \neq 0$. Because $\Lip(\nabla F;B(a,2R)) = M$, we have $|\nabla F(x)| = |\nabla F(x) - \nabla F(0)| \leq M|x|$. Hence, for $x \in B(a,R)$, $\big(x-\nabla F(x)/M \big)\in B(a,2R)$. From \eqref{TI} evaluated at $z=x-\frac{\nabla F(x)}{M}$ and the previous inequality,
    \begin{align*}
        F\left(x-\nabla F(x)/M\right) &\leq F(x) + \left\langle \nabla F(x), x-\nabla F(x)/M - x \right\rangle + \frac{M}{2}\left|x-\nabla F(x)/M-x\right|^2\\
        &=F(x) -|\nabla F(x)|^2/M+|\nabla F(x)|^2/(2M)\\
        &< \left(\frac{1}{2M}-\frac{1}{M}+\frac{1}{2M} \right) |\nabla F(x)|^2 <0,
    \end{align*}
    but this contradicts our deduction that $F(z)\geq 0$ for $z \in B(a,2R) \setminus \{ 0 \}$. Thus, the lemma holds.
\end{proof}
For a convex function $F \in C^{1,1}_c(\R^n)$, we have $\Lip(\nabla F;\R^n) < \infty$, implying the following:

\begin{cor} \label{neccor}
    Let $F \in C_c^{1,1}(\R^n)$ be convex; then $J_xF \sim_{M} J_yF$ for all $x,y \in \R^n$, where $M=\Lip(\nabla F; \R^n)$.
\end{cor}

\subsection{Estimates for $C_c^{1,1}$-Convex Extension of $1$-Jets}

Let $E \subset \R^n$ be closed, and suppose the Whitney field $(\gamma_x)_{x \in E}$ satisfies $\gamma_x \sim_M \gamma_y$ for all $x, y \in E$. Then $(\gamma_x)_{x \in E}$ satisfies the hypotheses of Theorem \ref{azagra0}, and we deduce there exists a convex function $F \in C^{1,1}_c(\R^n)$ satisfying $J_xF = \gamma_x$ for all $x \in E$ and $\Lip(\nabla F) \leq M$.

Under the same hypotheses (i.e., that the Whitney field $(\gamma_x)_{x \in E}$ satisfies $\gamma_x \sim_M \gamma_y$ for all $x, y \in E$), we can add (\ref{a}) to (\ref{b}) and apply the Cauchy-Schwartz inequality to deduce 
\begin{align*}
    |\nabla \gamma_y - \nabla \gamma_x|| y-x| \geq \langle \nabla \gamma_y - \nabla \gamma_x, y-x \rangle \geq \frac{1}{M}|\nabla \gamma_y - \nabla \gamma_x|^2,
\end{align*}
and thus, $|\nabla \gamma_y - \nabla \gamma_x| \leq M|y-x|$. The non-negativity of (\ref{b}) implies
\begin{align*}
    \gamma_x(x) - \gamma_y(x) \leq \big(\gamma_x(x) - \gamma_y(x)\big) + \big( \gamma_y(y) - \gamma_x(y) \big) = \langle \nabla \gamma_y - \nabla \gamma_x, y-x \rangle \leq M|y-x|^2.
\end{align*}
Similarly, from the non-negativity of (\ref{a}), we see if $\gamma_x \sim_M \gamma_y$,
\begin{align}
\gamma_y(y) - \gamma_x(y) \leq M|y-x|^2. \label{wh1}    
\end{align}
This implies the following:
\begin{remark}
    If the Whitney field $(\gamma_x)_{x \in E}$ satisfies 
$\gamma_x \sim_M \gamma_y$ for all $x,y \in E$, and $\sup_{x \in E} \{|\gamma_x(x)|\} + \sup_{x \in E} \{ |\nabla \gamma_x| \} \leq M $, the polynomials $(\gamma_x)_{x \in E} \in Wh(E)$ satisfy the hypotheses of Whitney's Extension Theorem for $C^{1,1}(\R^n)$ with a constant $M$, implying there exists a function $G \in C^{1,1}(\R^n)$ satisfying $J_xG = \gamma_x$ for all $x \in E$, and $\|G\|_{C^{1,1}(\R^n)} \leq C(n)M$. But the function $G$ need \underline{not} be convex (see e.g. \cite{st} for this version of Whitney's Extension Theorem).
\end{remark}

Next we construct an example of constants $M \geq \eta>0$, a set $E \subset \R$, a function $f:E \to \R$, and a choice of Whitney field $(\gamma_x)_{x \in E}  \in Wh(E)$ satisfying $\gamma_x \in \Gamma^E_\eta(x)$ and $\gamma_x \sim_M \gamma_y$ for all $x,y \in E$, such that there does \underline{not} exist an $\eta$-strongly convex function $F \in C^{1,1}_c(\R)$ satisfying $\Lip(\nabla F) \leq M$ and $J_xF = \gamma_x$ for all $x \in E$.

\textbf{Example:} Let $E = \{0,1\} \subset \R$, $\eta  \in (0,1/4)$, $M=1$, $f(0)=0$ and $f(1)=\eta/2$. Then the polynomials $\gamma_0 \in \mathcal{R}_0, \gamma_1 \in \mathcal{R}_1$ defined as $\gamma_0(x) = 0$, and $\gamma_1(x) = \frac{\eta}{2} + 2\eta(x-1)$ satisfy
\begin{align}
    &\gamma_0(0) - \gamma_1(0) = \frac{3 \eta}{2} \geq \frac{\eta}{2}|1-0|^2, \label{i1}\\
    & \gamma_0(0) - \gamma_1(0) = \frac{3 \eta}{2} \geq \frac{1}{2}|2\eta-0|^2, \label{i2}\\
    &\gamma_1(1) - \gamma_0(1) = \frac{\eta}{2} = \frac{\eta}{2} |1-0|^2, \text{ and} \label{i3}\\
    &\gamma_1(1) - \gamma_0(1) = \frac{\eta}{2}\geq \frac{1}{2}|2\eta-0|^2. \label{i4}
\end{align}
Inequality (\ref{i1}) implies $\gamma_0 \in \Gamma^E_\eta(0)$; (\ref{i3}) implies $\gamma_1 \in \Gamma^E_\eta(1)$; and together (\ref{i2}) and (\ref{i4}) imply $\gamma_0 \sim_1 \gamma_1$. 

Notice any $\eta$-strongly convex extension of $\gamma_0$ must lie above the function $g: \R \to \R$ defined as $g(x) := \frac{\eta}{2}|x|^2$. But $g(1) = \gamma_1(1) = \frac{\eta}{2}$ and $g'(1) = \eta < 2 \eta = \nabla \gamma_1$. Thus, there is no $\eta$-strongly convex function $F \in C^{1,1}_c(\R^n)$ satisfying $J_0F=\gamma_0 $, $ J_1F= \gamma_1$. However, the following proposition implies that there is an $\eta/p$-strongly convex function satisfying these conditions, for any $p>1$.


\begin{prop} \label{scprop}
    Let $E^* \subset \R^n$ be closed, $f^*: E^* \to \R$, and $M \geq \eta > 0$. Suppose $(\gamma_x)_{x \in E^*}$ satisfies $\gamma_x \in \Gamma^{E^*}_\eta(x; f^*)$ and $\gamma_x \sim_M \gamma_y$ for all $x,y \in E^*$. Let $p,q \in (1, \infty)$ satisfy $\frac{1}{p}+\frac{1}{q}=1$. Then there exists an ${\eta/p}$-strongly convex function $F \in C^{1,1}_c(\R^n)$ satisfying $J_xF = \gamma_x$ for all $x \in E^*$ and $\Lip(\nabla F) \leq qM +{\eta/p} \leq (q+1)M$.
\end{prop}

We will turn to the proof of Proposition \ref{scprop} momentarily. The following lemma explains the reduction in the strong convexity constant $\eta$ of an extension, under these hypotheses: 

\begin{lem} \label{flexsc}
Let $E \subset \R^n$ be closed, $f: E \to \R$, and $M \geq \eta > 0$. Suppose $(\gamma_x)_{x \in E}$ satisfies $\gamma_x \in \Gamma^E_\eta(x;f)$ and $\gamma_x \sim_M \gamma_y$ for all $x,y \in E$. Let $p,q \in (1, \infty)$ satisfy $\frac{1}{p}+\frac{1}{q}=1$. For $x \in E$, let $P_x \in \mathcal{R}_x$ satisfy $P_x(x) = f(x) - \frac{\eta}{2{p}}|x|^2$ and $\nabla P_x = \nabla \gamma_x - \frac{\eta}{{p}} x$. Then $P_x \sim_{qM} P_y$ for all $x,y \in E$.
\end{lem}

\begin{proof}
Let $x,y \in E$. By translation (by $y$) and subtraction of the affine function $\gamma_y$, we can assume $y = 0$ and $\gamma_y = 0$. By hypothesis, $\gamma_x(x) = f(x)$; thus, $\gamma_x(z) = f(x) + \langle \nabla \gamma_x, z-x \rangle$. Then $\gamma_x \sim_M \gamma_y$ implies
\begin{align}
    &f(x) \geq \frac{1}{2M} |\nabla \gamma_x|^2, \text{ and} \label{v1} \\
    &-f(x) +\langle \nabla \gamma_x, x \rangle \geq \frac{1}{2M}|\nabla \gamma_x|^2. \label{v2}
\end{align}
The conditions $\gamma_x \in \Gamma^E_\eta(x)$, $\gamma_y \in \Gamma^E_\eta(y)$ imply 
\begin{align}
    &f(x) \geq \frac{\eta}{2}|x|^2, \label{v3} \\
    &-f(x) +\langle \nabla \gamma_x, x \rangle \geq \frac{\eta}{2}|x|^2,\text{ and} \label{v4} \\
    & |\nabla \gamma_x||x| \geq \langle \nabla \gamma_x,x \rangle \geq \eta|x|^2, \label{v5}
\end{align}
where the final inequality comes from adding (\ref{v3}) to (\ref{v4}) and applying the Cauchy Schwartz inequality to the inner product. We want to prove that $P_x \sim_{qM} P_y$. By our normalizations, $P_y(y) = f(y) -\frac{\eta}{2{p}} |y|^2=0$, and $\nabla P_y = \nabla \gamma_y - \frac{\eta}{{p}}y =0$, so $P_y = 0$. Thus, we want to show
\begin{align}
    &f(x) - \frac{\eta}{2{p}} |x|^2 \geq \frac{1}{2qM} |\nabla \gamma_x - \frac{\eta}{{p}} x|^2, \text{ and} \label{cw1} \\
    &-f(x) + \frac{\eta}{2{p}}|x|^2 - \langle \nabla \gamma_x - \frac{\eta}{{p}} x, -x\rangle \geq \frac{1}{2qM}|\nabla \gamma_x - \frac{\eta}{{p}} x|^2. \label{cw2}
\end{align}
From (\ref{v1}) and (\ref{v3}), we have
\begin{align}
     &f(x) - \frac{\eta}{2{p}}|x|^2 \geq \frac{1}{2qM}|\nabla \gamma_x|^2 + \bigg(1-\frac{1}{q}\bigg)\frac{\eta}{2}|x|^2 - \frac{\eta}{2{p}}|x|^2,  \label{v6}
\end{align}
and from (\ref{v2}) and (\ref{v4}),
\begin{align}
    -f(x) + \frac{\eta}{2{p}}|x|^2 - \langle \nabla \gamma_x - \frac{\eta}{{p}} x, -x\rangle &=-f(x) + \langle \nabla \gamma_x , x\rangle - \frac{\eta}{2{p}} |x|^2 \nonumber \\
    &\geq \frac{1}{2qM}|\nabla \gamma_x|^2 + \left( 1-\frac{1}{q}\right)\frac{\eta}{2}|x|^2- \frac{\eta}{2{p}} |x|^2 \label{v7}.
\end{align}
We next bound the terms on the right-hand side of (\ref{v6}) and (\ref{v7}) (which are the same). Because $\frac{1}{p}+\frac{1}{q}=1$, we have:
\begin{align*}
    \frac{1}{2qM}|\nabla \gamma_x|^2 + &\bigg(1-\frac{1}{q}\bigg)\frac{\eta}{2}|x|^2 - \frac{\eta}{2{p}}|x|^2 = \frac{1}{2qM}|\nabla \gamma_x|^2 \\
    &= \frac{1}{2qM}|\nabla \gamma_x-\frac{\eta}{{p}} x|^2 -\frac{\eta^2}{2q{p}^2M}|x|^2 + \frac{\eta}{q{p}M} \langle \nabla \gamma_x, x \rangle \\
    &\overset{(\ref{v5})}\geq \frac{1}{2qM}|\nabla \gamma_x-\frac{\eta}{{p}} x|^2 - \frac{\eta^2}{2q{p}^2M}|x|^2 + \frac{\eta^2}{q{p}M} |x|^2 \\
    &\geq \frac{1}{2qM}|\nabla \gamma_x-\frac{\eta}{{p}} x|^2,
\end{align*}
where the last inequality uses the fact that $2p \geq 1$. Thus, we have proven (\ref{cw1}) and (\ref{cw2}).

\end{proof}

\begin{proof}[Proof of Proposition \ref{scprop}]
We apply Lemma \ref{flexsc}: For $x \in E$, let $P_x \in \mathcal{R}_x$ satisfy $P_x(x) = f(x) - {\frac{\eta}{2p}}|x|^2$ and $\nabla P_x = \nabla \gamma_x - {\frac{\eta}{p}} x$. Then $P_x \sim_{qM} P_y$ for all $x,y \in E$. We apply Theorem \ref{azagra0} to the polynomials $(P_x)_{x \in E}$, to produce a convex function $G \in C^{1,1}_c(\R^n)$ satisfying $J_xG = P_x$ for all $x \in E$  and $\Lip(\nabla G) \leq qM$. We define $F \in C^{1,1}(\R^n)$ as $F(x) = G(x)+{\frac{\eta}{2p}}|x|^2$. Because $G$ is convex, $F$ is ${\frac{\eta}{p}}$-strongly convex; $\Lip(\nabla F) \leq qM+{\frac{\eta}{p}}\leq (q+1)M$, and $J_xF = \gamma_x$ for all $x \in E$. This completes the proof of the proposition.

\end{proof}

\subsection{Relevant Convex Subsets of $\mpp$} \label{sec:ksets}

In this section we introduce convex sets of jets that are relevant to the $C^{1,1}_c(\R^n)$ extension problem. We establish the basic properties of these sets in the two lemmas below. We invoke Helly's theorem to prove Lemma \ref{k1} which states that the convex sets are non-empty when $f : E \rightarrow \R $ satisfies a finiteness hypothesis, as in the assumptions of Theorem \ref{scthm}.

For the rest of this section, we fix a finite set $E \subset \R^n$ and a function $f:E \to \R$. For $M>\eta>0$, and $T \subset E$, we define the set $\widetilde{\Gamma}^1_\eta(T,M;f) \subset Wh(T)$ as
\begin{align*}
    &\widetilde{\Gamma}^1_\eta(T,M;f) := \bigcap_{ z \in E } \left\{ (P_x)_{x \in T} \in Wh(T): \begin{aligned}  &\exists \text{ an } \eta \text{-strongly convex function } F \in C^{1,1}_c(\R^n) \text{ s.t. } \\
    &\Lip(\nabla F) \leq M, \; J_xF = P_x \; \forall x \in T, \text{ and } F|_{\{z\} \cup T} = f|_{\{z\} \cup T}
    \end{aligned} \right\}.
\end{align*}

Again, where the function $f$ is evident, we will not write $f$; i.e., we write $\widetilde{\Gamma}^1_\eta(T,M)$ in place of $\widetilde{\Gamma}^1_\eta(T,M;f)$.

We establish the basic containments for these convex sets in the following. 

\begin{lem}
    Let $E \subset \R^n$ be finite and $f:E \to \R$. For all $T \subset E$, we have
    \begin{align}
    &\widetilde{\Gamma}^1_\eta(T,M_0) \subset \widetilde{\Gamma}^1_{\eta}(T,M)  &&(M>M_0>\eta> 0), \text{ and}\label{gc2.5}\\
    &\widetilde{\Gamma}^1_\eta(T,M) \subset \widetilde{\Gamma}^1_{\eta_0}(T,M) &&(M>\eta>\eta_0>0).  \label{gc3}
\end{align}
Let $(P_x)_{x \in T} \in \widetilde{\Gamma}^1_{\eta}(T,M)$; then
\begin{align}
    P_x \in \Gamma_\eta^E(x) \quad \quad  \text{for all } x \in T. \label{gc5}
\end{align}
\end{lem}

\begin{proof}
    Properties (\ref{gc2.5}) and (\ref{gc3}) are immediate from the definition of the set $\widetilde{\Gamma}^1_{\eta}(T,M)$. Let $(P_x)_{x \in T} \in \widetilde{\Gamma}^1_\eta(T,M)$ and $x \in T$. For $y \in E \setminus \{x\}$, there exists an $\eta$-strongly convex function $F^y \in C^{1,1}_c(\R^n)$ such that $J_xF = P_x$ and $F^y|_{\{x,y\}}=f|_{\{x,y\}}$. Thus, $P_x(y) + \frac{\eta}{2} |y-x|^2 \leq f(y)$. Because this is true for all $y \in E \setminus \{x\}$, we deduce $P_x \in \Gamma^E_\eta(x)$.

\end{proof}

\begin{lem} \label{convk}
    Let $E \subset \R^n$ be finite, $f:E \to \R$, and $M> \eta>0$; for $T \subset E$, the set $\widetilde{\Gamma}^1_\eta(T,M) \subset Wh(T)$ is convex.
\end{lem}

\begin{proof}
For $T\subset E$, $z \in E$, define $K(T,z) \subset Wh(T)$ as
\begin{align*}
    K(T,z):=\left\{ (P_x)_{x \in T} \in Wh(T): \begin{aligned}  &\exists \text{ an } \eta \text{-strongly convex function } F \in C^{1,1}_c(\R^n) \text{ s.t. } \\
    &\Lip(\nabla F) \leq M, \; J_xF = P_x \; \forall x \in T, \text{ and } F|_{\{z\} \cup T} = f|_{\{z\} \cup T}
    \end{aligned} \right\}.
\end{align*}
Let $F_1, F_2 \in C^{1,1}_c(\R^n)$ be $\eta$-strongly convex functions satisfying $\Lip(\nabla F_i) \leq M$ for $i=1,2$. Then $F_t: = tF_1 + (1-t)F_2$ is $\eta$-strongly convex and satisfies $\Lip(\nabla F_t) \leq M$. Because $J_xF^t = tJ_xF_1 + (1-t)J_xF_2$, we see $(J_xF_1)_{x \in T}, (J_xF_2)_{x \in T} \in K(T,z)$ implies $(J_xF_t)_{x \in T} \in K(T,z)$. Thus, $K(T,z)$ is convex, and $\widetilde{\Gamma}^1_\eta(T,M) = \bigcap_{z \in E} K(T,z)$ is convex.

\end{proof}

\begin{lem} \label{k1}
    Let $E \subset \R^n$ be finite and $f:E \to \R$. Suppose $j,k \in \N$ satisfy $k \geq (n+2)j+1$, and for every $S \subset E$ satisfying $\#S \leq k$, there exists an $\eta$-strongly convex function $F^S \in C^{1,1}_c(\R^n)$ satisfying $F^S|_S=f|_S$ and $\Lip(\nabla F^S) \leq M$. Then for every $T \subset E$ satisfying $\#T \leq j$, $\widetilde{\Gamma}^1_\eta(T,M) \neq \emptyset$.
\end{lem}

\begin{proof}
In the previous lemma, we proved that $\widetilde{\Gamma}^1_\eta(T,M) = \bigcap_{z \in E  } K(T,z)$ is the intersection of a family of $\big((n+1)|T|\big)$-dimensional convex sets. Here we show that given $T \subset E$, every subfamily of $\big((n+1)|T|+1\big)$ of the convex sets $K(T,z)$ has a non-empty intersection. Then we can apply Helly's Theorem (Theorem \ref{helly}) to prove the intersection of the entire family is non-empty.

Let $T \subset E$ satisfy $\#T \leq j$. Let $z_1, \dots, z_{(n+1)j+1}$ be points of $E  $; then  $S := \{z_1,\cdots,z_{(n+1)j+1}\}$ is contained in $E$ and $\#(S \cup T) \leq (n+2)j+1 \leq k$. By applying the hypothesis of the lemma to the set $S \cup T$, we deduce there exists an $\eta$-strongly convex function $F^{S \cup T} \in C^{1,1}_c(\R^n)$ satisfying $F^{S \cup T}|_{S \cup T} = f|_{S \cup T}$ and $\Lip( \nabla F^{S \cup T}) \leq M$. Thus,
\begin{align*}
    (J_xF^{S \cup T})_{x \in T} &\in \bigcap_{i = 1}^{(n+1)j+1} \left\{ (P_x)_{x \in T} \in Wh(T): \begin{aligned}  &\exists \text{ an } \eta \text{-strongly convex function } F \in C^{1,1}_c(\R^n) \text{ s.t. } \\
    &\Lip(\nabla F) \leq M, \; J_xF = P_x \; \forall x \in T, \text{ and } F|_{\{z_i\} \cup T} = f|_{\{z_i\} \cup T}
    \end{aligned} \right\} \\
    &\quad \quad = \bigcap_{i=1}^{(n+1)j+1} K(T,z_i).
\end{align*}
Thus, the intersection of any $\big((n+1)j+1\big)$-element subfamily of $\{K(T,z) : z \in E  \}$ is non-empty. We can then apply Helly's Theorem to conclude $\widetilde{\Gamma}^1_\eta(T,M)  = \bigcap_{z \in E  } K(T,z)$ is non-empty.
\end{proof}

\subsection{Lipschitz Selection of Gradient Vectors}

We want to make a Lipschitz selection of gradient vectors from the sets $\left( \{\nabla \gamma: \gamma \in {\Gamma}^E_{{\eta}}(x;f) \}\right)_{x \in E}$. To do so, we adapt Theorem \ref{selthm0} to our setting:

\begin{prop}\label{selthm}
    There exist $k^\#_s\in \N$ and $C^\#>0$ such that the following holds: Let $E \subset \R^n$ be finite. For each $x \in E$, let $K(x) \subset \R^n$ be a closed convex set. Suppose that for each $S \subset E$ with $\#S \leq k^\#_s$, there exists a Lipschitz function $F^S : \R^n \rightarrow \R^n$ satisfying $\Lip (F^S) \leq 1$ and $F^S(x) \in K(x)$ for all $x \in S$. Then there exists a $C^{0,1}$ (bounded and Lipschitz) function $F : \R^n \rightarrow \R^n$ with $\Lip ( F) \leq C^\#$, such that $F(x) \in K(x)$ for all $x \in E$.
\end{prop}

\begin{proof}
    Let $k^\#_s:= k^\#_s(n,D)=k^\#_s(n,n)$ from Theorem \ref{selthm0}. Because $E \subset \R^n$ is finite, there exists $R_0>0$ such that $E \subset B(0,R_0)$. For each $S \subset E$ with $\#S \leq k^\#_s$, let $F^S : \R^n \rightarrow \R^n$ be a Lipschitz function satisfying $\Lip (F^S) \leq 1$ and $F^S(x) \in K(x)$ for all $x \in S$. Each $F^S$ is locally bounded, and there are finitely many $S$, so we can let $A_0 \in (0,\infty)$ be a uniform upper bound on $\sup_{x \in B(0,R_0)}|F^S(x)|$ for every $S \subset E$ with $\#(S) \leq k^\#_s$. Thus,
    
    $$\bigcup_{\substack{ S \subset E,\\ \#S \leq k^\#_s}} \left\{F^S(x):x \in B(0,R_0)\right\} \subset B(0,A_0),$$ 
    and because $\Lip(F^S) \leq 1$, $\|F^S\|_{C^0(B(A_0 + 2 R_0))} \leq 2A_0 +R_0$ for any $S \subset E$ satisfying $\#S \leq k^\#_s$.
    
    Let the function $\theta \in C^{\infty}(\R^n)$ be a smooth bump function satisfying $0 \leq \theta\leq 1$, $\theta|_{B(0,R_0)} = 1|_{B(0,R_0)}$, $\supp(\theta) \subset B(0, A_0 + 2R_0)$, and $|\nabla \theta (x)| \leq \frac{C_0}{(A_0 + R_0)}$ for $x \in \R^n$. For $S \subset E$, $\#S \leq k^\#_s$, let the function $\bar{F}^S: \R^n \to \R^n$ be $\bar{F}^S(x): = F^S(x) \theta(x)$. Then $\bar{F}^S \in C^{0,1}(\R^n;\R^n)$ satisfies
    \begin{align*}
        &\|\bar{F}^S\|_{C^0(\R^n)} \leq \|F^S\|_{C^0(B(A_0 + 2 R_0))}  \leq 2A_0 + R_0, \\
        & \bar{F}^S(x) \in K(x) \text{ for all }x \in S, \text{ and}\\
        &\Lip(\bar{F}^S) \leq \Lip(F^S) +\|F^S\|_{C^0(B(A_0 + 2 R_0))}\cdot \Lip(\theta) \leq 1 + \frac{(2A_0 + R_0)C_0}{A_0 + R_0} \leq 1+2C_0.
    \end{align*}
    Fix $N: = 2(A_0+R_0)$. Let $E' \subset \R^n$ and $K'(x) \subset \R^n$ for $x\in E$ be 
    \begin{align*}
        &E' =\frac{1}{N}E := \{y/N: y\in E \}, \text{ and} \\
        &K'(x):= \frac{1}{N(2+4C_0)} K(x):= \left\{\frac{y}{N(2+4C_0)}: y\in K(x) \right\}.
    \end{align*}
    For $S' \subset E'$ satisfying $\#S' \leq k^\#_s$, the set $NS'$ is contained in $E$. 
    Let the function $G^{S'}: \R^n \to \R^n$ be
    \begin{align*}
        G^{S'}(x):=
            \frac{1}{N(2+4C_0)}\bar{F}^{NS'}(Nx).
    \end{align*}
    Then $|G^{S'}(x)| \leq 1/2$ for all $x \in \R^n$, and $\Lip(G^{S'}) \leq \frac{1}{2+4C_0} \Lip(\bar{F}^{NS'}) \leq 1/2$, implying the function $G^{S'}$ satisfies:
    \begin{align*}
        & \|G^{S'}\|_{C^{0,1}(\R^n, \R^n)} \leq 1 \text{ and }\\
        &G^{S'}(x) \in K'(Nx) \text{ for }x \in S'.
    \end{align*}
    This result holds for all $S' \subset E'$ satisfying $\#S \leq k^\#_s$. Thus, the convex sets $(K'(Nx))_{x \in E'}$ satisfy the hypotheses of Theorem \ref{selthm0} with $(n,D)=(n,n)$; applying Theorem \ref{selthm0}, we deduce there exists $G \in C^{0,1}(\R^n,\R^n)$ satisfying $G(x) \in K'(Nx)$ for all $x \in E'$ and $\|G\|_{C^{0,1}(\R^n, \R^n)} \leq C_0^\#$, where $C_0^\#:=C^\#(n,n)$ from Theorem \ref{selthm0}. Let the function $F: \R^n \to \R^n$ be $F(x):=N(2+4C_0) G(\frac{x}{N})$; $F$ satisfies $F(x) \in K(x)$ for all $x \in E$ and $\Lip(F) \leq C_0^\#(2+4C_0)=:C^\#$. This completes the proof of Proposition \ref{selthm}.
\end{proof}

\section{Main Results in Dimension $n \geq 1$  (Theorem \ref{scthm})} \label{sec:scthm}

We begin by proving a version of Theorem \ref{scthm} for finite $E$. Using a compactness argument, we will then extend the result to arbitrary $E$.

\begin{THM} \label{finthm}
    Let $E \subset \R^n$ be finite, $M>\eta>0$, $f: E \to \R$, and ${k^\#} = k^\#_s(n+2) +1$, where $k^\#_s$ is the constant from Proposition \ref{selthm}. Suppose for all $S \subset E$ satisfying $\#S \leq {k^\#}$, there exists an $\eta$-strongly convex function $F^S \in C^{1,1}_c(\R^n)$ satisfying $F^S|_S=f|_S$ and $\Lip(\nabla F^S) \leq M$. Then for all $p,q \in (1, \infty)$ satisfying $\frac{1}{p}+\frac{1}{q}=1$, there exists an ${\frac{\eta}{p}}$-strongly convex function $F \in C^{1,1}_c(\R^n)$ satisfying $F|_E = f|_E$, and $\Lip(\nabla F) \leq C_3M^2/\eta$, where $C_3 = (C^\# )^2(q+1) $, and $C^\#$ is the constant in Proposition \ref{selthm}.   
\end{THM}

\begin{proof}

By applying Lemma \ref{k1} with $j:= k^\#_s$ and $k := k^\#_s(n+2) +1$, we see for every $T \subset E$ satisfying $\#T \leq k^\#_s$, $\widetilde{\Gamma}^1_\eta(T,M) \neq \emptyset$. From the definition of the set $\widetilde{\Gamma}^1_\eta(T,M)$, we deduce for $T \subset E$ satisfying $\#T \leq k^\#_s$, there exists an $\eta$-strongly convex function $F^T \in C^{1,1}_c(\R^n)$ satisfying $\Lip(\nabla F^T) \leq M$, $F|_T=f|_T$, and $(J_xF^T)_{x \in T} \in \widetilde{\Gamma}^1_\eta(T,M)$. In light of \eqref{gc5}, $J_xF^T \in \Gamma_\eta^E(x)$ for all $x \in T$. We summarize this result:
\begin{align}
    &\text{For all } T \subset E \text{ satisfying } \#T \leq k^\#_s \text{ there exists an }\eta \text{-strongly convex function }F^T \in C^{1,1}_c(\R^n)\nonumber \\
    &\text{satisfying } \Lip(\nabla F^{T}) \leq M \text{ and } J_xF^{T} \in \Gamma^E_{\eta}(x) \; \forall x \in T.  \label{FH2}
\end{align} 

For $x \in E$, $\eta>0$, let $\bar{\Gamma}^E_{\eta}(x) \subset \R^n$ be
\[
\bar{\Gamma}^E_{\eta}(x) : = \{ \nabla P: P \in \Gamma^E_{\eta}(x) \}.
\]
(See \eqref{Gamma_summary} for the definition of $\Gamma^E_\eta(x)$.) For $x \in E$, the set $\Gamma^E_\eta(x)$ is a convex subset of $\mathcal{R}_x$, so the set $\bar{\Gamma}^E_{\eta}(x)$ is a convex subset of $\R^n$. 

Let $S \subset E$ satisfy $\#S \leq k^\#_s$. Define $G^S \in C^{0,1}(\R^n, \R^n)$ as $G^S(x):= \nabla F^S (x)$, where $F^S \in C^{1,1}_c(\R^n)$ is the function produced by applying (\ref{FH2}) to the set $S$. Then $G^S(x) \in \bar{\Gamma}_{\eta}^E(x)$ for all $x \in S$ and $\Lip(G^S) \leq M$. We apply Proposition \ref{selthm} to produce $G \in C^{0,1}(\R^n,\R^n)$ satisfying $G(x) \in \bar{\Gamma}_{\eta}^E(x)$ for all $x \in E$ and 
\begin{align}
    \Lip(G) \leq C^\# M. \label{lipc}
\end{align}

For $x \in E$, define $\gamma_x \in \mathcal{R}_x$ as the polynomial satisfying $\gamma_x(x)=f(x)$ and $\nabla \gamma_x = G(x)$. Immediately, $\gamma_x \in \Gamma^E_{\eta}(x)$, so for all $x,y \in E$ 
\begin{align}
    &\gamma_x(x) - \gamma_y(x) \geq \frac{\eta}{2} |x - y|^2, \text{ and} \label{sca}\\
    &\gamma_y(y) - \gamma_x(y) \geq \frac{\eta}{2} |x - y|^2. \label{scb}
\end{align}
In light of (\ref{lipc}), $|\nabla \gamma_x - \nabla \gamma_y| \leq C^\#  M |x-y|$. In combination with (\ref{sca}) and (\ref{scb}), we deduce for all $x,y \in E$,
\begin{align*}
    &\gamma_x(x) - \gamma_y(x) \geq \frac{\eta}{2(C^\# M)^2} |\nabla \gamma_x - \nabla \gamma_y|^2, \text{ and} \\
    &\gamma_y(y) - \gamma_x(y) \geq \frac{\eta}{2(C^\# M)^2} |\nabla \gamma_x - \nabla \gamma_y|^2,
\end{align*}
implying $\gamma_x \sim_{\frac{M^2}{\eta}(C^\#)^2} \gamma_y$ for all $x,y \in E$. Thus, we can apply Proposition \ref{scprop} to produce an ${\frac{\eta}{p}}$-strongly convex function $F \in C^{1,1}_c(\R^n)$ satisfying $F|_E=f|_E$ and $\Lip (\nabla F) \leq (C^\#)^2(q+1) M^2 /\eta$, where $C^\#$ is the constant from Proposition \ref{selthm}. This completes the proof of Theorem \ref{finthm}.

\end{proof}

\begin{proof}[Proof of Theorem \ref{scthm}]

Fix $p,q \in (1, \infty)$ satisfying $\frac{1}{p}+\frac{1}{q}=1$. Let $E \subset \R^n$ be compact. There exists $R>1$ such that $E \subset B(0,R)$. For $A>0$, let $\mathcal{B}(A) \subset C^1(B(0,2R))$ be
\begin{align*}
    \mathcal{B}(A) = \{ F \in C^{1,1}_c(B(0,2R)): F \text{ is }\frac{\eta}{p} \text{-strongly convex, and } \|F\|_{C^{1,1}(B(0,2R))} \leq A \}.
\end{align*}
The set $\mathcal{B}(A)$ is closed in the $C^1(B(0,2R))$-topology. For any $A>0$, $\mathcal{B}(A)$ is also bounded and equicontinuous in the $C^1(B(0,2R))$-topology,  implying by the Arzel\`a-Ascoli Theorem that $\mathcal{B}(A)$ is compact. 

Let $E'$ be a countable dense subset of $E$, and let $(E_i)_{i \in \N}$ be an increasing sequence of sets satisfying for $i \in \N$, $E_i \subset E'$,  $\# E_i < \infty$, and $\bigcup_{i \in \N} E_i = E'$. By assumption, for all $S \subset E_i \subset E$ satisfying $\#S \leq {k^\#}$, there exists an $\eta$-strongly convex function $F^S \in C^{1,1}_c(\R^n)$ satisfying $F^S|_S=f|_S$ and $\Lip(\nabla F^S) \leq M$. 

Therefore, for each $i \in \N$, we can apply Theorem \ref{finthm} to produce $F_i \in C^{1,1}_c(\R^n)$, satisfying
\begin{align*}
    &F_i \text{ is an } {\frac{\eta}{p}} \text{-strongly convex function,}\\
    &F_i|_{E_i} = f|_{E_i}, \text{ and} \\
    & \Lip\big(\nabla F_i\big) \leq  C_3\frac{M^2}{\eta},
\end{align*}
where $C_3 =(C^\#)^2(q+1)$. Restricting the domain of $F_i$ to $B(0,2R)$, we see for $A$ large enough, $(F_i)_{i \in \N} \subset \mathcal{B}(A)$.

By the compactness of $\mathcal{B}(A)$, there exists a convergent subsequence $(F_{i_j})_{j \in \N} \to \bar{F} \in \mathcal{B}(A)$ in the $C^1$ topology. The function $\bar{F}$ satisfies $\bar{F} \in C^{1,1}_c(B(0,2R))$,
\begin{align*}
    &\bar{F} \text{ is an } {\frac{\eta}{p}} \text{-strongly convex function,}\\
    &\bar{F}|_{E'} = f|_{E'}, \text{ and} \\
    & \Lip\big(\nabla \bar{F}; B(0,2R)\big) \leq C_3\frac{M^2}{\eta},
\end{align*}
where the last inequality follows because $\Lip\big(\nabla F_{i_j}\big) \leq C_3\frac{M^2}{\eta}$ for all $j \in \N$. Further, because this convergence is uniform, $\bar{F}|_E=f|_E$.

We plan to apply Proposition \ref{scprop} to $(J_x\bar{F})_{x \in B(0,R)}$, so we verify the hypotheses of the proposition with $E^*:=\overline{B(0,R)}$ and $f^*:=\bar{F}$: By Lemma \ref{neclem}, $J_x\bar{F} \sim_{\Lip(\nabla \bar{F};B(0,2R))} J_y\bar{F}$ for all $x,y \in \overline{B(0,R)}$. Recall that in \eqref{Gamma_summary} we defined $\Gamma^{\overline{B(0,R)}}_\eta(x;\bar{F}) \subset Wh(\overline{B(0,R)})$ for $x \in \overline{B(0,R)}$ as
\[
\Gamma^{\overline{B(0,R)}}_\eta(x;\bar{F}):=\bigcap_{y \in \overline{B(0,R)} \setminus \{x\}} \{ P \in \mathcal{R}_x: P(x) = \bar{F}(x) \text{ and } P(y) + \frac{\eta}{2}|y-x|^2 \leq \bar{F}(y) \}.
\]
Because $\bar{F}$ is ${\frac{\eta}{p}}$-strongly convex on $B(0,2R)$, $J_x\bar{F} \in \Gamma_{{\frac{\eta}{p}}}^{\overline{B(0,R)}}(x;\bar{F})$ for all $x \in \overline{B(0,R)}$. 

We apply Proposition \ref{scprop} to $(J_x \bar{F})_{x \in \overline{B(0,R)}}$ to produce an ${\frac{\eta}{p^2}}$-strongly convex function $F \in C^{1,1}_c(\R^n)$ satisfying
$J_xF = J_x \bar{F}$ for all $x \in \overline{B(0,R)}$ and $\Lip(\nabla F) \leq (q+1)\Lip (\nabla \bar{F};B(0,2R))\leq C_1q^2 \frac{M^2}{\eta}$. Because $E \subset B(0,R)$ and $J_x \bar{F} \in \Gamma_{{\frac{\eta}{p}}}^{\overline{B(0,R)}}(x;\bar{F})$ for $x \in B(0,R)$, this implies $F(x)=f(x)$ for all $x \in E$. This completes the proof of Theorem \ref{scthm}.

\end{proof}

\section{Main Results in Dimension $n=1$ (Theorem \ref{1dthmsc})} \label{sec:1dthm}

We begin by proving a version of Theorem \ref{1dthmsc} for finite $E$  and $\eta = 0$. In Section \ref{sec:fin1d}, we complete the proof of Theorem \ref{1dthmsc}.

\begin{THM} \label{1dthmfin}
Let $E \subset \R$ be finite, and let the function $f: E \to \R$. Suppose that for every $S \subset E$ satisfying $\#S \leq 5$, there exists a function $F^S \in C^{1,1}_c(\R)$ satisfying $F^S|_S=f|_S$ and $\Lip(\nabla F^S) \leq M$. Then there exists a function $F \in C^{1,1}_c(\R)$ satisfying $F|_E = f|_E$ and $\Lip (\nabla F) \leq 2M$.
\end{THM}

\subsection{Technical Tools in Dimension $n=1$}

Let $E \subset \R$ be finite and $f: E \to \R$. Assuming the validity of a finiteness hypothesis on $f$, as per Theorem \ref{1dthmfin}, our aim is to find $(\gamma_x)_{x \in E} \in Wh(E)$ that satisfies $\gamma_x(x) =f(x)$ and $\gamma_x \sim_M \gamma_y$ for all $x,y \in E$, with a uniform constant $M$; then we can apply Theorem \ref{azagra0} to produce a convex extension of $f$ in $C^{1,1}_c(\R)$. In our first result, we deduce a transitivity property of the relation $\sim_M$ in one dimension. According to this, we only need to confirm the compatibility of $\gamma_x$ at adjacent points of $E$. 

\begin{lem} \label{translm}
    Let $x,y,z \in \R$ satisfy $x<y<z$, and suppose $\gamma_x \in \mathcal{R}_x$, $\gamma_y \in \mathcal{R}_y$, $\gamma_z \in \mathcal{R}_z$ satisfy $\gamma_x  \sim_M \gamma_y$, and $\gamma_y  \sim_M \gamma_z$; then $\gamma_x  \sim_M \gamma_z$.
\end{lem}

\begin{proof}
Suppose $x<y<z$, $\gamma_x  \sim_M \gamma_y$, and $\gamma_y  \sim_M \gamma_z$; we have
\begin{align}
    \frac{1}{2M} |\nabla \gamma_z - \nabla \gamma_x|^2 &= \frac{1}{2M} \big( |\nabla \gamma_z - \nabla \gamma_y|^2 +|\nabla \gamma_y - \nabla \gamma_x|^2 +2 \langle \nabla \gamma_z -\nabla \gamma_y,\nabla \gamma_y - \nabla \gamma_x \rangle\big) \nonumber \\
    &\leq \big( \gamma_y(y) - \gamma_z(y) \big) + \big( \gamma_x(x) - \gamma_y(x) \big)+ \frac{1}{M} \big\langle \nabla \gamma_z -\nabla \gamma_y,\nabla \gamma_y - \nabla \gamma_x \big\rangle \nonumber \\
    &=\gamma_x(x) - \gamma_z(x) + \bigg\langle \nabla \gamma_z - \nabla \gamma_y, x - y + \frac{1}{M} \big(\nabla \gamma_y - \nabla \gamma_x\big)  \bigg\rangle. \label{c}
\end{align}
Because $x<y$ and $\gamma_x  \sim_M \gamma_y$, we have $\frac{\nabla \gamma_y - \nabla \gamma_x}{M} \leq y-x$, which implies $x - y + \frac{1}{M} (\nabla \gamma_y - \nabla \gamma_x)\leq 0$. Because $\nabla \gamma_z - \nabla \gamma_y>0$, the last term in (\ref{c}) must be negative, implying $ \frac{1}{2M} |\nabla \gamma_z - \nabla \gamma_x|^2 \leq \gamma_x(x) - \gamma_z(x)$. The proof of the inequality $ \frac{1}{2M} |\nabla \gamma_z - \nabla \gamma_x|^2 \leq \gamma_z(z) - \gamma_x(z)$ follows analogously.
\end{proof}

\begin{remark}
    A transitivity result relying only on the configuration of points cannot be transferred to higher dimensions ($n>1$) because we cannot ensure $\langle \nabla \gamma_z - \nabla \gamma_y, x - y + \frac{1}{M} \big(\nabla \gamma_y - \nabla \gamma_x\big)\rangle$ is non-positive without further hypotheses on $\gamma_x, \gamma_y,$ and $\gamma_z$. Even if $x,y,z$ are ordered points on a line in $\R^n$ (i.e. there exists $t \in (0,1)$ such that $y = tx + (1-t)z$), the quantity $\langle \nabla \gamma_z - \nabla \gamma_y, x - y + \frac{1}{M} \big(\nabla \gamma_y - \nabla \gamma_x\big)\rangle$ need not be non-positive. However, if we assume $x,y,z \in \R^n$ satisfy $\langle y-x, z-y \rangle >0$ \underline{and} $(\nabla \gamma_z - \nabla \gamma_y) = \lambda (\nabla \gamma_y - \nabla \gamma_x)$ for $\lambda>0$, then we can show $\gamma_x \sim_M \gamma_y$, $\gamma_y \sim_M \gamma_z$ implies $\gamma_x \sim_M \gamma_z$.
\end{remark}

\begin{lem} \label{areflm}
    Let $E \subset \R$ be finite, $f:E \to \R$, and $M > 0$. For distinct $x, y \in E$, let $\gamma_x \in \Gamma^0(x)$ satisfy 
    \begin{equation}
        \label{as0} 0 \leq \langle D_{xy}^f - \nabla \gamma_x, y-x \rangle \leq \frac{M}{2} |y-x|^2.
    \end{equation}
If $\gamma_y^x \in \Gamma^0(y)$ is defined by 
    \begin{align*}
    &\nabla \gamma_y^x = \nabla \gamma_x + \sqrt{2M\langle D^f_{xy} -\nabla \gamma_x, y-x\rangle} && \text{ if }x<y, \text{ and}\\
    &\nabla \gamma_y^x = \nabla \gamma_x - \sqrt{2M\langle \nabla \gamma_x-D^f_{xy} , x-y\rangle} && \text{ if } x>y,
    \end{align*}
then $ \gamma_x \sim_M \gamma_y^x$. 
\end{lem}

\begin{proof}
Suppose $x<y$. By definition of $\gamma_y^x$ and the fact that $\gamma_x \in \Gamma^0(x)$, $\gamma_y^x \in \Gamma^0(y)$, we have
\[
    \frac{1}{2M} \left| \nabla \gamma_y^x - \nabla \gamma_x \right|^2 =  \langle D^f_{xy} - \nabla \gamma_x, y - x \rangle = \gamma_y^x(y) - \gamma_x(y),
\]
implying (\ref{b}) holds with equality for $\gamma_y:= \gamma_y^x$ and $\gamma_x: = \gamma_x$. 

We next show that (\ref{a}) holds for $\gamma_y:= \gamma_y^x$ and $\gamma_x: = \gamma_x$, proving $\gamma_y^x \sim_M \gamma_x$. From (\ref{as0}) and the definition of $\gamma_y^x$, we have
\begin{align}
    \nabla \gamma_y^x &\leq \nabla \gamma_x + M(y - x) \nonumber \\
    &\leq \nabla \gamma_x + M(y-x) + \sqrt{-2M  \langle D^f_{xy} -\nabla \gamma_x ,y-x \rangle + M^2(y-x)^2}, \label{as2}
\end{align}
where the non-negativity of the term inside the square root also follows from (\ref{as0}). Further,
\begin{align}
\notag{}
    \nabla \gamma_y^x &=  \nabla \gamma_x + \sqrt{2M \langle D^f_{xy} - \nabla \gamma_x , y - x\rangle} \\
    &\geq \nabla \gamma_x + M(y-x) - \sqrt{-2M  \langle D^f_{xy} -\nabla \gamma_x ,y-x \rangle + M^2(y-x)^2}.
    \label{as4}
\end{align}
The latter inequality follows from the observation that for $\omega>0$, $\sqrt{4 \omega t} \geq 2 \omega - \sqrt{4 \omega^2-4\omega t} $ for all $t \in [0, \omega]$. Let $\omega := \frac{M}{2}(y - x)$ and $t := D^f_{xy} - \nabla \gamma_x$ ((\ref{as0}) ensures $t \in [0,\omega]$), and the result follows.

For $\gamma_x \in \Gamma^0(x)$, $\gamma_y \in \Gamma^0(y)$, and $x<y$, (\ref{a}) holds if
$$
f(x) - f(y) - \langle \nabla \gamma_y, x-y \rangle \geq \frac{1}{2M}|\nabla \gamma_x -\nabla \gamma_y|^2.
$$
This is expression is equivalent to a quadratic equation in the variable $\nabla \gamma_{y}$:
\begin{align}
(\nabla \gamma_{y})^2 + (\nabla \gamma_{y}) \left(-2 \nabla \gamma_{x}+ 2M(x-y) \right) +\left( (\nabla \gamma_{x})^2 + 2M(f(y) - f(x))\right) \leq 0. \label{as3}
\end{align}
The discriminant $\Delta \in \R$ of this quadratic equation is non-negative thanks to (\ref{as0}); indeed,
\begin{align*}
   \Delta &= \left(-2 \nabla \gamma_{x}+ 2M(x-y) \right)^2 - 4\left( (\nabla \gamma_{x})^2 + 2M(f(y) - f(x))\right) \\
    & = 4M \left( -2 \big( f(y) - f(x) - \langle \nabla \gamma_{x} ,y-x \rangle \big) + M(y-x)^2 \right) \\
    &= -8 M \langle D^f_{xy} - \nabla \gamma_x , y-x \rangle + 4M^2 (y-x)^2 \geq 0.
\end{align*}
Thus, \eqref{as3} is equivalent to 
\begin{align*}
    \nabla \gamma_{y} \in \bigg[ &\nabla \gamma_{x} + M(y-x) - \sqrt{-2M \langle D_{xy}^f-  \nabla \gamma_{x} ,y-x \rangle + M^2(y-x)^2}, \\
    & \nabla \gamma_{x} + M(y-x) + \sqrt{-2M \langle D_{xy}^f-  \nabla \gamma_{x} ,y-x \rangle + M^2(y-x)^2} \bigg],
\end{align*}
which is valid for $\gamma_y := \gamma_y^x$ thanks to \eqref{as2} and \eqref{as4}. This completes the proof of (\ref{a}), and with it the proof that $\gamma_x \sim_M \gamma_y^x$

The proof that if $x>y$ and $\gamma_y^x \in \Gamma^0(y)$ satisfies $\nabla \gamma_y^x = \nabla \gamma_x - \sqrt{2M\langle \nabla \gamma_x-D^f_{xy} , x-y\rangle}$, then $\gamma_x \sim_M \gamma_y^x$ follows analogously.
\end{proof}

\subsection{Proof of Theorem \ref{1dthmfin}}

\begin{proof}[Proof of Theorem \ref{1dthmfin}]

Because $E$ is finite, we enumerate $E= \{x_1, x_2,\dots, x_N\}$, with $x_1<x_2 < \dots <x_N$. We may assume $N > 5$, else the result is trivial. For distinct $i, j \in \{1,\dots, N\}$, let 
\[
D_{i,j}:=D^f_{x_i x_j}= \frac{f(x_j) - f(x_i)}{x_j-x_i}.
\]
By the finiteness hypothesis, the restriction of $f$ to any $3$ consecutive points of $E$ admits a convex extension, hence
\begin{equation*}
D_{1,2} \leq D_{2,3} \leq \dots \leq D_{N-1,N}.
\end{equation*}

For $i \in \{ 1, \dots, N-1 \}$, let $P_i^\ell \in \Gamma^0(x_i)$, and for $i \in \{ 2, \dots, N \}$, let $P_i^r \in \Gamma^0(x_i)$ satisfy 
\begin{align}
    &\nabla P_1^\ell := D_{1,2}- \frac{M}{2} (x_2 - x_1), \nonumber \\
    &\nabla P_{i}^\ell := \max \left\{ D_{i-1,i}, D_{i,i+1}- \frac{M}{2}(x_{i+1} - x_{i}) \right\} &&(i \in \{2,\dots,N-1\}), \label{pil}\\
    & \nabla P_N^r := D_{N-1,N}+ \frac{M}{2} (x_N - x_{N-1}), 
    \text{ and}\nonumber\\
    &\nabla P_{i}^r := \min \left\{D_{i,i+1}, D_{i-1,i} + \frac{M}{2}(x_{i} - x_{i-1}) \right\} &&(i \in \{2,\dots,N-1\}) \label{pir}.
\end{align}
For $i \in \{2,\dots,N-1\}$, let $P_i^+, P_i^- \in \Gamma^0(x_i)$ satisfy
\begin{align}
    &\nabla P_i^+ =  \nabla P_{i-1}^\ell + \sqrt{2M \langle D_{i-1,i} - \nabla P_{i-1}^\ell , x_i - x_{i-1}\rangle},\text{ and} \label{defp1} \\
    &\nabla P_i^- = \nabla P_{i+1}^r - \sqrt{2M\langle \nabla P_{i+1}^r - D_{i,i+1}, x_{i+1} - x_i\rangle} \label{defp2}.
\end{align}

We use the following monotonicity result in the estimates that follow.
\begin{lem} \label{mono}
    For $\omega>0$,  the function $h(t) = -t 
    +\sqrt{4\omega t}$ is an increasing function on $[0,\omega]$. Thus, for $0 \leq t_1 \leq t_2\leq \omega$, we have $h(t_1) \leq h(t_2)$, and if $t_1 < t_2$ then $h(t_1) < h(t_2)$.
\end{lem}

\begin{proof}
    Trivial.
\end{proof}

For $i \in \{2, \dots, N-1\}$, by definition (\ref{pil}), we have  $0 \leq D_{i-1,i} - \nabla P_{i-1}^\ell \leq  \frac{M}{2}(x_i -x_{i-1})$. So, we can apply Lemma \ref{mono} with $t_1: =D_{i-1,i} - \nabla P_{i-1}^\ell$ and $t_2:= \omega:=\frac{M}{2}(x_i -x_{i-1})$; then  $h(t_1) \leq h(t_2)$, so
\begin{align}
    \nabla P_i^+ &= \nabla P_{i-1}^\ell + \sqrt{2M \langle D_{i-1,i} - \nabla P_{i-1}^\ell, x_{i}-x_{i-1} \rangle} \nonumber \\
    &\leq D_{i-1,i}-\frac{M}{2}(x_i -x_{i-1}) + \sqrt{2M \left\langle \frac{M}{2}(x_i -x_{i-1}), x_{i}-x_{i-1} \right\rangle} \nonumber \\
    &= D_{i-1,i}+\frac{M}{2}(x_i -x_{i-1}) \quad \quad \quad \quad (i \in \{2, \dots, N-1\}).
\label{T1}
\end{align}
Similarly, by the definition (\ref{pir}) of $P_{i+1}^r$, we have $0 \leq \nabla P_{i+1}^r - D_{i,i+1} \leq \frac{M}{2}(x_{i+1}-x_i)$, so we can apply Lemma \ref{mono} with $t_1:=\nabla P_{i+1}^r - D_{i,i+1}$ and $t_2:= \omega:= \frac{M}{2}(x_{i+1}-x_i)$; then  $-h(t_1) \geq -h(t_2)$, so
\begin{align}
    \nabla P_i^- &=\nabla P_{i+1}^r - \sqrt{2M\langle \nabla P_{i+1}^r - D_{i,i+1}, x_{i+1} - x_i\rangle} \nonumber \\
    &\geq D_{i,i+1} + \frac{M}{2}(x_{i+1} - x_i) - \sqrt{2M \left\langle  \frac{M}{2}(x_{i+1} - x_i), x_{i+1}-x_i \right\rangle} \nonumber \\
    &= D_{i,i+1} - \frac{M}{2}(x_{i+1} - x_{i}) \quad \quad \quad \quad (i \in \{2, \dots, N-1\}). \label{T2}
\end{align}

\begin{lem} \label{lem:pilpir}
    For $i \in \{ 2, \dots, N-1 \}$, $P_{i-1}^\ell \in \Gamma^0(x_{i-1})$ satisfies $\nabla P_{i-1}^\ell \geq \nabla P $ for all $P \in \Gamma^0(x_{i-1})$ satisfying $ P \sim_M P_i^+ $. Similarly, $P_{i+1}^r \in \Gamma^0(x_{i+1})$ satisfies $\nabla P_{i+1}^r \leq \nabla P$ for all $P \in \Gamma^0(x_{i+1})$ satisfying $P \sim_M P_i^-$.
\end{lem}

\begin{proof} 
Let $P \in \Gamma^0(x_{i-1})$ satisfy  $P \sim_M P_i^+$, then from inequality (\ref{a}) with $\gamma_x := P_i^+$ and $\gamma_y:= P$ we have
\begin{align}
\nabla P + \sqrt{2M\langle D_{i-1,i} - \nabla P, x_i -x_{i-1}\rangle} \geq \nabla P_i^+. \label{e1}
\end{align}
If, in addition, $\nabla P> \nabla P_{i-1}^\ell$, then $0 \leq D_{i-1,i} - \nabla P  < D_{i-1,i}-\nabla P_{i-1}^\ell  \leq \frac{M}{2}(x_i -x_{i-1})$, thanks to (\ref{pil}). So, we can apply Lemma \ref{mono} with $t_1 := D_{i-1,i} - \nabla P$, $t_2:=D_{i-1,i}-\nabla P_{i-1}^\ell$, and $\omega:=\frac{M}{2}(x_i -x_{i-1})$ to see that $h(t_1)<h(t_2)$, i.e.,
\begin{align*}
\nabla P + \sqrt{2M\langle D_{i-1,i} - \nabla P, x_i -x_{i-1}\rangle} < \nabla P_{i-1}^\ell + \sqrt{2M\langle D_{i-1,i} - \nabla P_{i-1}^\ell, x_i -x_{i-1}\rangle} = \nabla P_i^+.
\end{align*}
But this contradicts \eqref{e1}. Thus, if $P\sim_M P_i^+$, we must have $\nabla P \leq \nabla P_{i-1}^\ell$. The proof that $\nabla P_{i+1}^r \leq \nabla P$ for all $P \in \Gamma^0(x_{i+1})$ satisfying $P \sim_M P_i^-$ follows analogously.
\end{proof}

\begin{lem} \label{lem:symg}
    Suppose that for every $S \subset E$ satisfying $\#S \leq 5$, there exists a function $F^S \in C^{1,1}_c(\R)$ satisfying $F^S|_S=f|_S$ and $\Lip(\nabla F^S) \leq M$. Then $\max \{\nabla P_i^-, D_{i-1,i} \} \leq  \min \{ \nabla P_i^+,D_{i,i+1} \}$ for $i \in \{ 2, \dots, N-1 \}$.
\end{lem}

\begin{proof}

Let the sets $S_i \subset E$ ($i \in \{1, \dots, N\}$) be defined by $S_i: = \{x_{i-2},x_{i-1}, x_i, x_{i+1}, x_{i+2} \}$ for $i \in \{3, \dots, N-2\}$, and $S_1, S_2:= S_3$, $S_N, S_{N-1}:=S_{N-2}$. By the finiteness hypothesis, for $i \in \{1, \dots,N\}$, there exists a function $F^{S_i} \in C^{1,1}_c(\R)$, satisfying $F^{S_i}|_{S_i} = f|_{S_i}$ and $\Lip(\nabla F^{S_i}) \leq M$. 
Because $F^{S_i}|_{S_i} = f|_{S_i}$,
\begin{align*}
    D_{j-1,j} = \frac{F^{S_i}(x_{j}) - F^{S_i}(x_{j-1})}{x_{j}-x_{j-1}} \quad \quad (i \in \{2, \dots, N-1\}, j \in \{i, i+1\}).
\end{align*}
Hence, the convexity of $F^{S_i}$ implies
\begin{equation}
    D_{i-1,i} \leq \nabla J_{x_i}F^{S_i} \leq D_{i,i+1} \quad\quad (i \in \{2,\dots,N-1\}).\label{d1}
\end{equation}

We claim that $\nabla P_i^- \leq \nabla J_{x_i}F^{S_i} \leq \nabla P_i^+$ for $i \in \{ 2, \dots, N-1 \}$. In combination with \eqref{d1}, this implies $\max \{\nabla P_i^-, D_{i-1,i} \} \leq \nabla J_{x_i}F^{S_i} \leq \min \{ \nabla P_i^+,D_{i,i+1} \}$, proving the lemma. 
    
For $i \in \{ 2, \dots, N\}$, Corollary \ref{neccor} implies $J_{x_{i-1}}F^{S_i} \sim_M J_{x_{i}}F^{S_i}$. Letting $\gamma_x: = J_{x_{i}}F^{S_i} $ and $\gamma_y: = J_{x_{i-1}}F^{S_i}$ in \eqref{a}, we see
\begin{align}
    \nabla J_{x_i}F^{S_i} \leq \nabla J_{x_{i-1}}F^{S_i} + \sqrt{2M \langle D_{i-1,i} - \nabla J_{x_{i-1}}F^{S_i}, x_i - x_{i-1}\rangle} \quad \quad  (i \in \{ 2, \dots, N\}).\label{d4}
\end{align}

By Taylor's inequality (\ref{TI}) and the convexity of $F^{S_i}$, $0 \leq F^{S_i}(x_{i-2}) - J_{x_{i-1}}F^{S_i}(x_{i-2})  \leq \frac{M}{2} (x_{i-1} - x_{i-2})^2$ for $i \in \{ 3, \dots, N\}$. Dividing by the positive quantity $(x_{i-1} - x_{i-2})$, we see that 
\begin{equation}\label{newline1}
    0 \leq \nabla J_{x_{i-1}}F^{S_i}-D_{i-2,i-1}  \leq \frac{M}{2} (x_{i-1} - x_{i-2}) \quad \quad ( i \in \{3, \dots, N\}). 
\end{equation}
Similarly, $0 \leq F^{S_i}(x_i) - J_{x_{i-1}} F^{S_i}(x_i) \leq \frac{M}{2} (x_i - x_{i-1})^2$ for $i \in \{2,\dots,N\}$, which implies
\begin{equation} \label{newline2}
0 \leq D_{i-1,i} - \nabla J_{x_{i-1}} F^{S_i} \leq \frac{M}{2} (x_i - x_{i-1}) \quad \quad ( i \in \{2,. \dots, N \}).
\end{equation}
By combining (\ref{newline1}) and (\ref{newline2}), and applying the definition \eqref{pil} of $\nabla P_{i-1}^\ell$, we have
\begin{align}
    &\nabla J_{x_{i-1}}F^{S_i} \geq \max \{ D_{i-1,i} - \frac{M}{2} (x_{i} - x_{i-1}), D_{i-2,i-1} \} = \nabla P_{i-1}^\ell \quad \quad ( i \in \{3, \dots, N\}). \label{d3}
\end{align}
When $i = 2$, (\ref{newline2})  reads as
$$
0 \leq D_{1,2} - \nabla J_{x_1}F^{S_2} \leq  \frac{M}{2}(x_2 - x_1) = D_{1,2} - \nabla P_1^\ell,
$$ 
where the last equality is by the definition of $\nabla P_1^\ell$.
Together with \eqref{d3}, we have  
$$0 \leq D_{i-1,i} - \nabla J_{x_{i-1}}F^{S_i} \leq D_{i-1,i} - \nabla P_{i-1}^\ell \leq \frac{M}{2}(x_i - x_{i-1}) \qquad (i \in \{2, \dots, N\}).
$$
Thus, we can apply Lemma \ref{mono} with $t_1: = D_{i-1,i} - \nabla J_{x_{i-1}}F^{S_i}$, $t_2:=D_{i-1,i} - \nabla P_{i-1}^\ell$, and $\omega:= \frac{M}{2}(x_i - x_{i-1})$ to see $h(t_1) \leq h(t_2)$, and in combination with \eqref{d4},
\begin{align*}
\nabla J_{x_i}F^{S_i} &\leq \nabla J_{x_{i-1}}F^{S_i} + \sqrt{2M \langle D_{i-1,i} - \nabla J_{x_{i-1}}F^{S_i}, x_i - x_{i-1}\rangle} \\
&\leq \nabla P_{i-1}^\ell + \sqrt{2M \langle D_{i-1,i} - \nabla P_{i-1}^\ell, x_i - x_{i-1} \rangle} = \nabla P_{i}^+ \quad \quad (i \in \{2, \dots, N \}).
\end{align*}
By an analogous argument, we deduce $\nabla J_{x_i}F^{S_i} \geq \nabla P_i^-$ for $i \in \{ 1, \dots, N-1 \}$. This completes the proof of the claim, and as described, the lemma.
\end{proof}

We are prepared to choose $\gamma_i \in \Gamma^0(x_i)$ for $i \in \{1, \dots, N\}$. We use the compatibility condition in Lemma \ref{areflm} to inform our choice of derivative for $\gamma_1 \in \Gamma^0(x_1)$ (chosen so that $\gamma_1 \sim_M \gamma_2$) and $\gamma_N \in \Gamma^0(x_N)$ (chosen so that $\gamma_{N-1} \sim_M \gamma_N$). Let $(\gamma_i)_{i=1}^N \in Wh(E)$ be the unique Whitney field of polynomials satisfying 
\begin{align}
    & \gamma_i \in \Gamma^0(x_i), \text{ and} \nonumber \\
    &\nabla \gamma_i = \begin{cases}
        \frac{1}{2} \big( \max \{\nabla P_i^-, D_{i-1,i} \}+\min \{ \nabla P_i^+,D_{i,i+1} \}\big) & i \in \{ 2, \dots,N-1\} \\
        \nabla \gamma_2 - \sqrt{2M \langle \nabla \gamma_2 - D_{1,2}, x_2 - x_1 \rangle} & i =1 \\
        \nabla \gamma_{N-1} + \sqrt{2M \langle D_{N-1,N} - \nabla \gamma_{N-1} , x_N -x_{N-1} \rangle} & i=N.
    \end{cases} \label{choosegam}
\end{align}
As a result of Lemma \ref{lem:symg}, for $i \in \{2, \dots, N-1\}$, we have
\begin{align}
    \max \{\nabla P_i^-, D_{i-1,i} \} \leq \nabla \gamma_i \leq \min \{ \nabla P_i^+,D_{i,i+1} \}. \label{symg}
\end{align}
Consequently, $\nabla \gamma_2 - D_{1,2} \geq 0$, and $D_{N-1,N} - \nabla \gamma_{N-1} \geq 0$, ensuring $\gamma_1$ and $\gamma_2$ are well-defined in \eqref{choosegam}.

In the next several results, we verify additional basic inequalities satisfied by $(\gamma_i)_{i=1}^N$. 

From the previous lemma, we have $\nabla \gamma_i \in [D_{i-1,1},D_{i,i+1}]$ for all $i \in \{2,\cdots,N-1\}$. Because the sequence of divided differences is non-decreasing, we have $\nabla \gamma_2 \leq \cdots \leq \nabla \gamma_{N-1}$. By inspection of the definitions of $\nabla \gamma_1$ and $\nabla \gamma_N$ in (\ref{choosegam}), we obtain the following result:
\begin{cor} \label{cor:grad}
    The polynomials $(\gamma_i)_{i=1}^N$ defined in (\ref{choosegam}) satisfy that their gradients are non-decreasing: 
    $$\nabla \gamma_1 \leq \nabla \gamma_2 \leq \dots \leq \nabla \gamma_N.$$
\end{cor}

\begin{lem}
    For $i \in \{ 3, \dots, N-1\}$, the polynomials $(\gamma_i)_{i=1}^N$ satisfy
    \begin{align}
    &\nabla P_{i-1}^\ell \leq \nabla \gamma_{i-1}, \text{ and}\label{g2}  \\
    &\nabla \gamma_{i} \leq \nabla P_{i}^r. \label{g3} 
\end{align}
\end{lem}
\begin{proof}
    By inequality \eqref{symg}, $\nabla \gamma_{i-1} \geq \max \{ \nabla P_{i-1}^-, D_{i-2,i-1} \} \geq D_{i-2,i-1}$. If $\nabla P_{i-1}^\ell = D_{i-2, i-1}$, then this implies $\nabla \gamma_{i-1} \geq \nabla P_{i-1}^\ell$. Else, by definition of $P_{i-1}^\ell$, we have $\nabla P_{i-1}^\ell = D_{i-1,i} - \frac{M}{2}(x_i - x_{i-1})$. Then 
    \[
    \nabla \gamma_{i-1} \geq \max \{ \nabla P_{i-1}^-, D_{i-2,i-1} \} \geq \nabla P_{i-1}^- \geq D_{i-1,i} - \frac{M}{2}(x_i - x_{i-1}) = \nabla P_{i-1}^\ell,
    \]
    where the last inequality is from (\ref{T2}). We have proved \eqref{g2}. The proof of \eqref{g3} follows analogously. 
\end{proof}

By applying Theorem \ref{azagra0}, the next lemma will be used to complete the proof of Theorem \ref{1dthmfin}.

\begin{lem}
    The Whitney field $(\gamma_i)_{i=1}^N \in Wh(E)$ defined in (\ref{choosegam}) satisfies $\gamma_i \sim_{2M} \gamma_j$ for all $i,j \in \{1, \dots, N\}$. \label{1drel}
\end{lem}

\begin{proof}
In light of Lemma \ref{translm}, we only need to prove $\gamma_i \sim_{2M} \gamma_{i-1}$ for $i \in \{2, \dots, N\}$. 

For $i \in \{2, \dots, N-1 \}$, inequalities \eqref{symg}, \eqref{T1}, and \eqref{T2} imply
\begin{align}
    &0 \leq \nabla \gamma_i - D_{i-1,i} \leq \nabla P_i^+ - D_{i-1,i} \leq \frac{M}{2}(x_i - x_{i-1}), \text{ and} \label{R14} \\
    &0 \leq D_{i,i+1}-\nabla \gamma_i \leq D_{i,i+1} - \nabla P_i^- \leq \frac{M}{2}(x_{i+1} - x_i), \label{R15}
\end{align}
Thus, we can apply Lemma \ref{areflm} with $\gamma_x: = \gamma_2$ and $y:= x_1$ to see  we have $\gamma_1 \sim_{M} \gamma_2$. Likewise, inequality (\ref{R15}) allows us to apply Lemma \ref{areflm} with $\gamma_x: = \gamma_{N-1}$ and $y:= x_N$ to see $\gamma_{N-1} \sim_M \gamma_N$. Therefore, it suffices to show $\gamma_i \sim_{2M} \gamma_{i-1}$ for $i \in \{3, \dots, N-1\}$. We prove this by demonstrating that 
\begin{align}
    &f(x_i) - f(x_{i-1}) - \langle \nabla \gamma_{i-1}, x_i- x_{i-1} \rangle \geq \frac{1}{4M}|\nabla \gamma_i - \nabla \gamma_{i-1}|^2 &&(i \in \{3,\dots, N-1\}), \label{aa1} \text{ and} \\
    &f(x_{i}) - f(x_{i+1}) - \langle \nabla \gamma_{i+1}, x_{i}- x_{i+1} \rangle \geq \frac{1}{4M}|\nabla \gamma_i - \nabla \gamma_{i+1}|^2 &&(i \in \{2,\dots, N-2\}). \label{aa2}
\end{align}
Indeed, the fact that (\ref{aa1}) holds, and (\ref{aa2}) holds with $i$ replaced by $(i-1)$, implies that $\gamma_i \sim_{2M} \gamma_{i-1}$ for $i \in \{3,\cdots,N-1\}$.

We will next establish inequality  (\ref{aa1}) by splitting into the two cases below.
Let $i \in \{3, \dots, N-1\}$. 

\underline{Case 1:} Suppose that $\nabla P_{i-1}^-  \leq \nabla P_{i-1}^\ell$. From the definition of $P_{i-1}^\ell$ in (\ref{pil}), we have $D_{i-2,i-1} \leq \nabla P_{i-1}^\ell$. Together, these inequalities imply
\begin{align*}
\nabla \gamma_{i-1} = \frac{1}{2} \big( \max \{\nabla P_{i-1}^-, D_{i-2,i-1} \}+\min \{ \nabla P_{i-1}^+,D_{i-1,i} \}\big)
&\leq \frac{1}{2} \big( \nabla P_{i-1}^\ell +\min \{ \nabla P_{i-1}^+,D_{i-1,i} \}\big) \\
&\leq \frac{1}{2} \big( \nabla P_{i-1}^\ell +D_{i-1,i} \big).
\end{align*}
With inequality (\ref{g2}), we've shown $\nabla \gamma_{i-1} \in \bigg[\nabla P_{i-1}^\ell, \frac{\nabla P_{i-1}^\ell+D_{i-1,i}}{2}\bigg]$. In particular, $\nabla P_{i-1}^\ell \leq D_{i-1,i}$. In combination with (\ref{R15}), 
$$
0 \leq D_{i-1,i}- \frac{\nabla P_{i-1}^\ell+D_{i-1,i}}{2}\leq D_{i-1,i}-\nabla \gamma_{i-1} 
\leq \frac{M}{2}(x_i - x_{i-1}),
$$ 
we can apply Lemma \ref{mono} with $t_1: = D_{i-1,i}- \frac{\nabla P_{i-1}^\ell+D_{i-1,i}}{2}$, $t_2:= D_{i-1,i}-\nabla \gamma_{i-1} $, and $\omega:=M(x_i - x_{i-1})$; the map $h$ is increasing, and in particular $ h(t_2) \geq h(t_1)$, so
\begin{align}
    \nabla \gamma_{i-1} + &\sqrt{4M(f(x_i) - f(x_{i-1}) - \langle \nabla \gamma_{i-1}, x_i- x_{i-1} \rangle)} \nonumber \\
    &=\nabla \gamma_{i-1} +\sqrt{4M\langle D_{i-1,i}-\nabla \gamma_{i-1}, x_i- x_{i-1} \rangle} \nonumber\\
    &\geq  \frac{\nabla P_{i-1}^\ell+D_{i-1,i}}{2}+ \sqrt{4M \left\langle D_{i-1,i} - \frac{\nabla P_{i-1}^\ell+D_{i-1,i}}{2}, x_i- x_{i-1} \right\rangle} \nonumber \\
    &\geq \nabla P_{i-1}^\ell  + \sqrt{2M \langle D_{i-1,i} - \nabla P_{i-1}^\ell , x_i- x_{i-1} \rangle} = \nabla P_i^+, \label{g1}
\end{align}
where the last inequality follows because $D_{i-1,i} \geq \nabla P_{i-1}^\ell$. From inequality \eqref{symg} and Corollary \ref{cor:grad}, we have $\nabla P_i^+ \geq \nabla \gamma_i \geq \nabla \gamma_{i-1}$; we first use these inequalities and then  \eqref{g1} to bound $\frac{1}{4M}|\nabla \gamma_{i-1}-\nabla \gamma_i|^2  \leq \frac{1}{4M}|\nabla \gamma_{i-1}-\nabla P_i^+|^2 \leq f(x_i) - f(x_{i-1}) - \langle \nabla \gamma_{i-1}, x_i- x_{i-1} \rangle$, which is (\ref{aa1}).

\underline{Case 2:} Suppose $\nabla P_{i-1}^\ell < \nabla P_{i-1}^-$;  thus, by the definition of $P_{i-1}^\ell$ in \eqref{pil},  $D_{i-2,i-1} \leq \nabla P_{i-1}^\ell < \nabla P_{i-1}^-$. Therefore,
\begin{align*}
\nabla \gamma_{i-1} = \frac{1}{2} \big( \max \{\nabla P_{i-1}^-, D_{i-2,i-1} \}+\min \{ \nabla P_{i-1}^+,D_{i-1,i} \}\big)
&= \frac{1}{2} \big( \nabla P_{i-1}^- +\min \{ \nabla P_{i-1}^+,D_{i-1,i} \}\big) \\
&\leq \frac{1}{2} \big( \nabla P_{i-1}^- +D_{i-1,i} \big).
\end{align*}
From inequality \eqref{symg}, $\nabla \gamma_{i-1} \geq \max \{ \nabla P_{i-1}^-, D_{i-2,i-1} \} \geq \nabla P_{i-1}^-$. We have shown $\nabla \gamma_{i-1} \in \bigg[ \nabla P_{i-1}^-, \frac{\nabla P_{i-1}^- + D_{i-1,i}}{2}\bigg]$. In particular, $\nabla P_{i-1}^- \leq D_{i-1,i}$. In combination with (\ref{R15}), 
$$
0 \leq D_{i-1,i} - \frac{\nabla P_{i-1}^- + D_{i-1,i}}{2} \leq D_{i-1,i}-\nabla \gamma_{i-1} \leq D_{i-1,i}- \nabla P_{i-1}^- 
\leq \frac{M}{2}(x_{i} - x_{i-1}).
$$ 
Thus, we can apply Lemma \ref{mono} with
$t_1: = D_{i-1,i} - \frac{\nabla P_{i-1}^- + D_{i-1,i}}{2}  $, $t_2: =D_{i-1,i}-\nabla \gamma_{i-1} $, and $\omega:=M(x_{i} - x_{i-1})$ to see $h(t_2) \geq h(t_1)$, so that
\begin{align}
    \nabla \gamma_{i-1} + &\sqrt{4M \langle D_{i-1,i}-\nabla \gamma_{i-1},x_i - x_{i-1}\rangle}\nonumber \\
    &\geq \frac{\nabla P_{i-1}^-+D_{i-1,i}}{2} + \sqrt{4M\left\langle D_{i-1,i}- \frac{\nabla P_{i-1}^-+D_{i-1,i}}{2},x_i - x_{i-1}\right\rangle} \nonumber \\
    &\geq \nabla P_{i-1}^- + \sqrt{2M \left\langle D_{i-1,i}- \nabla P_{i-1}^-,x_i - x_{i-1}\right\rangle}. \label{e2}
\end{align}

According to (\ref{R15}), $0 \leq  D_{i-1,i} - \nabla P_{i-1}^- \leq \frac{M}{2}(x_{i} - x_{i-1})$. This verifies the hypothesis of Lemma \ref{areflm} for $\gamma_x:=P_{i-1}^-$ and $y:=x_i$. Thus, there exists $\tilde{\gamma}_i^{i-1} \in \Gamma^0(x_i)$ satisfying 
\[
\nabla \tilde{\gamma}_i^{i-1} = \nabla P_{i-1}^- + \sqrt{2M \left\langle D_{i-1,i}- \nabla P_{i-1}^-,x_i - x_{i-1}\right\rangle}
\]
and $\tilde{\gamma}_i^{i-1} \sim_M P_{i-1}^-$. In light of Lemma \ref{lem:pilpir}, we must have $\nabla \tilde{\gamma}_i^{i-1} \geq \nabla P_i^r$. Hence, continuing from \eqref{e2},
\begin{align}
    \nabla \gamma_{i-1} + \sqrt{4M \langle D_{i-1,i}-\nabla \gamma_{i-1},x_i - x_{i-1}\rangle} &\geq \nabla P_{i-1}^- + \sqrt{2M \left\langle D_{i-1,i}- \nabla P_{i-1}^-,x_i - x_{i-1}\right\rangle} \nonumber \\
    &=\nabla \tilde{\gamma}_i^{i-1}\geq \nabla P_i^r \geq \nabla \gamma_i, \label{f1}
\end{align}
where the last inequality follows from \eqref{g3}.
This is equivalent to (\ref{aa1}). We have exhausted all cases proving (\ref{aa1}). 

The proof of inequality (\ref{aa2}) follows analogously because of the symmetry of our choice of $\nabla \gamma_i :=\frac{1}{2} \big( \max \{\nabla P_i^-, D_{i-1,i} \}+\min \{ \nabla P_i^+,D_{i,i+1} \}\big)$ for $i \in \{ 2, \dots,N-1\}$ in light of Lemma \ref{lem:symg}. We summarize the proof briefly. Recall inequality \eqref{aa2} is
\begin{align*}
    f(x_{i}) - f(x_{i+1}) - \langle \nabla \gamma_{i+1}, x_{i}- x_{i+1} \rangle &=\langle \nabla \gamma_{i+1}-D_{i,i+1}, x_{i+1}- x_{i} \rangle\\ &\geq \frac{1}{4M}|\nabla \gamma_i - \nabla \gamma_{i+1}|^2 \quad (i \in \{2,\dots, N-2\}).
\end{align*}
First, we suppose $\nabla P_{i+1}^+ \geq \nabla P_{i+1}^r$; then $\nabla \gamma_{i+1} \in \left[  \frac{\nabla P_{i+1}^r + D_{i,i+1}}{2}, \nabla P_{i+1}^r\right]$. We use Lemma \ref{mono} and inequality \eqref{symg} to see (\ref{aa2}) holds under the assumption $\nabla P_{i+1}^+ \geq \nabla P_{i+1}^r$. Second, we suppose $\nabla P_{i+1}^+ < \nabla P_{i+1}^r$ and, therefore, $\nabla \gamma_{i+1} \in \bigg[ \frac{\nabla P_{i+1}^+ + D_{i,i+1}}{2}, \nabla P_{i+1}^+\bigg]$. We use Lemmas \ref{areflm}-\ref{lem:pilpir} and inequality \eqref{symg} to see (\ref{aa2}) holds, exhausting all cases and completing the proof of (\ref{aa2}).

This completes the proof of Lemma \ref{1drel}.

\end{proof}

By applying Theorem \ref{azagra0} to $(\gamma_i)_{i =1}^N \in Wh(E)$ and using that $\gamma_i(x_i) = f(x_i)$, we complete the proof of Theorem \ref{1dthmfin}.
\end{proof}

\subsection{Proof of Theorem \ref{1dthmsc}} \label{sec:fin1d}

Theorem \ref{1dthmsc} is an immediate consequence of the following theorem:

\begin{THM} \label{1dthm}
Let $E \subset \R$ be compact, the function $f: E \to \R$, and $M>0$. Suppose for every $S \subset E$ satisfying $\#S \leq {k_1^\#}=5$, there exists a convex function $F^S \in C^{1,1}_c(\R)$ satisfying $F^S|_S=f|_S$ and $\Lip(\nabla F^S) \leq M$. Then there exists a convex function $F \in C^{1,1}_c(\R)$ satisfying $F|_E = f|_E$ and $\Lip (\nabla F) \leq 2M$.
\end{THM}

To see Theorem \ref{1dthmsc} follows, we assume the hypotheses of Theorem \ref{1dthmsc}: Let $E \subset \R$ be compact, and let the function $f: E \to \R$. Suppose for every $S \subset E$ satisfying $\#S \leq {k_1^\#}=5$, there exists an $\eta$-strongly convex function $F^S \in C^{1,1}_c(\R)$ satisfying $F^S|_S=f|_S$ and $\Lip(\nabla F^S) \leq M$. Let $g: E \to \R$ be $g(x):= \frac{1}{1+\eta/M} (f(x) - \frac{\eta}{2}|x|^2)$, and for $S \subset E$ satisfying $\#S \leq {k_1^\#}=5$, let $G^S: \R \to \R$ be $G^S(x):= \frac{1}{1+\eta/M}(F^S(x) - \frac{\eta}{2} |x|^2)$. The function $G^S$ satisfies $G^S$ is convex, $G^S|_S = g|_S$, and $\Lip(\nabla G^S) \leq \frac{1}{1+\eta/M}(\Lip(\nabla F^S) + \eta) \leq M$. Thus, $g:E \to \R$ satisfies the hypotheses of Theorem \ref{1dthm}. Applying this theorem, we deduce there exists a convex function $G \in C^{1,1}_c(\R)$ satisfying $G|_E = g|_E$ and $\Lip (\nabla G) \leq 2M$. Let the function $F: \R \to \R$ be $F(x):= (1+\eta/M)G(x) + \frac{\eta}{2}|x|^2$. Then $F \in C^{1,1}_c(\R)$ is $\eta$-strongly convex and satisfies $\Lip(\nabla F) \leq 2M+3\eta$ and $F|_E=f|_E$. The conclusion of Theorem \ref{1dthmsc} follows.

Thus, our remaining task is to prove Theorem \ref{1dthm}.

\begin{proof}[Proof of Theorem \ref{1dthm}]

Let $E \subset \R$ be compact. There exists $R>1$ such that $E \subset B(0,R)$. For $A>0$, let $\mathcal{B}(A) \subset C^1(B(0,2R))$ be
\begin{align*}
    \mathcal{B}(A) = \{ F \in C^{1,1}_c(B(0,2R)): F \text{ is convex, and } \|F\|_{C^{1,1}(B(0,2R))} \leq A \}.
\end{align*}
The set $\mathcal{B}(A)$ is closed in the $C^1(B(0,2R))$-topology. For any $A>0$, $\mathcal{B}(A)$ is also bounded and equicontinuous in the $C^1(B(0,2R))$-topology,  implying by the Arzel\`a-Ascoli Theorem that $\mathcal{B}(A)$ is compact. 

Let $E'$ be a countable dense subset of $E$, and let $(E_i)_{i \in \N}$ be an increasing sequence of sets satisfying for $i \in \N$, $E_i \subset E'$,  $\# E_i < \infty$, and $\bigcup_{i \in \N} E_i = E'$. By assumption, for all $S \subset E_i \subset E$ satisfying $\#S \leq 5$, there exists an $\eta$-strongly convex function $F^S \in C^{1,1}_c(\R)$ satisfying $F^S|_S=f|_S$ and $\Lip(\nabla F^S) \leq M$.  We apply Theorem \ref{1dthmfin} to produce a convex function $F_i \in C^{1,1}_c(\R^n)$ satisfying $F_i|_{E_i}=f|_{E_i}$, and $\Lip (\nabla F_i) \leq 2M$. Restricting the domain of $F_i$ to $B(0,2R)$, we see for $A$ large enough, $F_i \in \mathcal{B}(A)$ for all $i \in \N$. By the compactness of $\mathcal{B}(A)$, there exists a convergent subsequence $(F_{i_k})_{k \in \N} \to \bar{F} \in \mathcal{B}(A)$. The limiting function $\bar{F}$ satisfies $\Lip(\nabla \bar{F};B(0,2R)) \leq 2M$ and $\bar{F}|_{E'}=f|_{E'}$, and because this convergence is uniform $\bar{F}|_E=f|_E$. Because $\bar{F} \in \mathcal{B}(A)$, $\bar{F} \in C^{1,1}_c(B(0,2R))$ is convex on $B(0,2R)$.

By Corollary \ref{neccor}, $J_x\bar{F} \sim_{\Lip(\nabla \bar{F};B(0,2R))} J_y\bar{F}$ for all $x,y \in \overline{B(0,R)}$. We apply Theorem \ref{azagra0} to $(J_x \bar{F})_{x \in \overline{B(0,R)}}$ to produce a convex function $F \in C^{1,1}_c(\R^n)$ satisfying $J_xF = J_x \bar{F}$ for all $x \in \overline{B(0,R)}$
(implying $F|_E=f|_E$) and $\Lip(\nabla F) \leq 2M$, completing the proof of Theorem \ref{1dthm}.

\end{proof}

\newpage
\bibliographystyle{amsplain}
\bibliography{main}

\end{document}